\documentclass[11pt,a4paper]{article}
\usepackage{a4wide,amsfonts,amsmath,latexsym,amsthm,amssymb,euscript,eufrak,graphicx,units,mathrsfs,setspace,stmaryrd}

\usepackage[pdftex]{color}
\definecolor{darkblue}{rgb}{0.3,0.3,0.7}

\usepackage{float}
\newfloat{figure}{H}{lof}
\floatname{figure}{\figurename}

\usepackage[french,english]{babel}
\usepackage[T1]{fontenc}
\usepackage[utf8]{inputenc}

\DeclareMathAlphabet{\eufrak}{U}{}{}{}  
\SetMathAlphabet\eufrak{normal}{U}{euf}{m}{n}
\SetMathAlphabet\eufrak{bold}{U}{euf}{b}{n}

\newtheorem{prop}{Proposition}[section]

\newtheorem{theorem}[prop]{Theorem}
\newtheorem{lemma}[prop]{Lemma}
\newtheorem{corollary}[prop]{Corollary}

\theoremstyle{definition}

\newtheorem{remark}[prop]{Remark}

\newtheorem{definition}[prop]{Definition}
\newtheorem{notation}[prop]{Notation}

\numberwithin{equation}{section}

\def\E{\mathbb{E}}
\def\P{\mathbb{P}}
\def\real{\mathbb{R}}

\def\F{\mathcal{F}}
\def\1{\textbf{1}}
\def\ind#1{\textbf{1}_{\{#1\}}}
\def\h{\hat{H}}

\newcommand{\eps}{\varepsilon}

\newcommand{\be}{\begin{equation}}
\newcommand{\ee}{\end{equation}}
\newcommand{\bde}{\begin{displaymath}}
\newcommand{\ede}{\end{displaymath}}
\newcommand{\beq}{\begin{eqnarray*}}
\newcommand{\eeq}{\end{eqnarray*}}
\newcommand{\beqa}{\begin{eqnarray}}
\newcommand{\eeqa}{\end{eqnarray}}
\newcommand{\bel }{\left\{\begin{array}{ll}}
\newcommand{\eel}{\cr \end{array} \right.}

\newcommand{\bex}{\begin{ex} \rm }
\newcommand{\eex}{\end{ex}}
{

\def\E{\mathbb E}
\def\F{{\cal F}}

\def\P{\mathbb P}

\def\cal#1{\mathcal{#1}}

\def\b{\textcolor{darkblue}}

\definecolor{ying}{rgb}{0.8, 0.0, 0.04}

\DeclareSymbolFontAlphabet{\mathrsfs}{rsfs}

\author{Caroline Hillairet\footnote{ENSAE  Paris, CREST UMR 9194,
5  avenue Henry Le Chatelier
91120 Palaiseau, France.  Email: \texttt{caroline.hillairet@ensae.fr}} \and Anthony R\'eveillac\footnote{INSA de Toulouse, IMT UMR CNRS 5219, Universit\'e de Toulouse, 135 avenue de Rangueil 31077 Toulouse Cedex 4 France. \; Email: \texttt{anthony.reveillac@insa-toulouse.fr}} \and Mathieu Rosenbaum\footnote{Ecole Polytechnique, CMAP UMR 7641,  Route de Saclay, 91120 Palaiseau, France. Email: \texttt{mathieu.rosenbaum@polytechnique.edu}}}

\title{An expansion formula for Hawkes processes and application to  cyber-insurance derivatives  \footnote{This research is supported by a grant of the French National Research Agency (ANR), “Investissements d’Avenir” (LabEx Ecodec/ANR-11-LABX-0047) and the Joint Research Initiative "Cyber Risk Insurance: actuarial modeling" with the partnership of AXA Research Fund. }}

\begin{document}

\maketitle
\allowdisplaybreaks

\begin{abstract}
\noindent
In this paper we provide an expansion formula for Hawkes processes which involves the addition of jumps at deterministic times to the Hawkes process in the spirit of the well-known integration by parts formula (or more precisely the Mecke formula) for Poisson functional. Our approach allows us to provide an expansion of the premium of a class of  cyber  insurance derivatives (such as reinsurance contracts including generalized Stop-Loss contracts) or risk management instruments (like Expected Shortfall) in terms of so-called shifted Hawkes processes. From the actuarial point of view, these processes can be seen as "stressed"  scenarios.  Our expansion formula for Hawkes processes enables us to provide lower and upper bounds on the premium (or the risk evaluation) of such  cyber  contracts and to quantify the surplus of premium compared to the standard modeling with a homogenous Poisson process.
\end{abstract}


\section{Introduction}
{In actuarial science, the classical Cramer-Lundberg model used} to describe the surplus process of an insurance portfolio relies on the assumptions of the claims arrival being modeled by a Poisson process,  and of independence among claim sizes and between claim sizes and claim inter-occurrence times. However, in practice those assumptions are often too restrictive and there is a need for more general models.
 A first generalization in the modeling of claims arrivals consists in  using Cox processes (also known as doubly stochastic Poisson processes), in the context of ruin theory such as in Albrecher and Asmussen (2006) [2], or for pricing stop-loss catastrophe insurance contract and catastrophe insurance derivatives such as in Dassios and Jang (2003) [13] and (2013) [14], or Hillairet et al. (2018) \cite{HJR2018}.  \\
  {Besides, self-exciting effects have been highlighted in cyber risk,  in favor of modeling the claims arrivals by a Hawkes process,  that is  adapted to model aftershocks of cyber attacks.  Such processes have been recently used in the cyber security field, for instance by  Peng et al. (2016) \cite{PENG2016} who focused on extreme cyber attacks rates.   Baldwin et al.  (2017) also studied in  \cite{BALDWIN2017}   the threats to 10 important IP services, using industry standard SANS data, and they claim that Hawkes processes provides the adequate modeling of cyber attacks into information systems because they capture both shocks and persistence after shocks that may form attack contagion.    In cyber insurance, the  statistical analysis of Bessy et al. (2020) \cite{bessy2020multivariate}  on   the public Privacy Rights Clearinghouse database highlights 
the ability of Hawkes models to capture self-excitation and interactions of data-breaches.  {Although the application in this paper focuses on cyber risk,  this methodology  can  be applied to other risks presenting self-exciting properties, such as credit risk.}\\
 Hawkes processes, which  have wide applications in many fields (such as seismology, finance,  neuroscience or social networks), are also beginning to get studied in actuarial sciences. } Magnusson  Thesis (2015)  \cite{magnusson2015risk} is dedicated to  Hawkes processes with exponential decay and their   application to insurance. 
Dassios and Zhao (2012)  \cite{dassios2012ruin}  consider the ruin problem in a model where the claims arrivals follow a Hawkes process with decreasing exponential kernel.  Stabile and  Torrisi (2010) \cite{stabile2010risk}   study the asymptotic behavior of infinite and finite horizon ruin probabilities, assuming  non-stationary Hawkes claims arrivals and under light-tailed conditions on the claims. Gao and Zhu (2018) \cite{gao2018large} establish large deviations results for Hawkes processes with an exponential kernel and   develop approximations for finite-horizon ruin probabilities.  {Swishchuk (2018) \cite{swishchuk2018risk}  applies limit theorems for risk model based on general compound Hawkes process to compute premium principles and ruin times.}\\
Rather than giving explicit computations for probabilities of ruin, our paper proposes to give pricing formulae of insurance contracts, such as Stop-Loss contracts.  
Stop-Loss is a non-proportional type of reinsurance and works similarly to excess-of-loss reinsurance. While excess-of-loss is related to single loss amounts, either per risk or per event, stop-loss covers are related to the total amount of claims in a year. The reinsurer pays the part of the total loss that exceeds a certain amount $\underline{ K}$. The reinsurer's liability is often limited to a given threshold. Stop-loss reinsurance offers protection against an increase in either or both severity and frequency of a company's loss experience.
Various approximations of stop-loss reinsurance premiums are described in literature, some of them assuming certain dependence structure, such as Gerber (1982) \cite{gerber1982}, Albers (1999) \cite{albers1999},  
De Lourdes Centeno (2005) \cite{de2005dependent} or  Reijnen et al.  (2005) \cite{reijnen2005}.\\
Stop-loss contracts are the paradigm of reinsurance contracts, but we aim at dealing with more general payoffs (of maturity $T$),  whose valuation involves the computation of the quantity of the form 
$$\E[K_T h(L_T)] \mbox{  where  }  \left\lbrace
\begin{array}{l}
   \mbox{  $K_T$ is the effective loss covered by the reinsurance company, } \\\\
\mbox{  $L_T$ is the loss quantity that activates the contract.  } \end{array}
\right.$$
For example, for stop loss contracts  $h(L_T)= \ind{L_T \geq \underline{ K}}$. 
Similarly, this methodology  can  be applied to valuation of credit derivatives, such as Credit Default Obligation  tranches. It also goes beyond the analysis of pricing and finds application in the computation of the expected
shortfall of contingent claims : the expected shortfall is a useful risk measure, that takes into account the size of the expected loss above the value at risk. We refer to \cite{HJR2018} for  more details on those analogies.\\
Our paper considers   {cyber} insurance contracts, with underlying a cumulative loss indexed by a Hawkes process. 
We propose to compute an explicit closed form pricing formula. Although Monte Carlo procedures are certainly the most efficient to compute numerically the premium of such general contracts, the closed form expansion formula we develop  allows to compute lower and upper bounds for the premium,  and to quantify the surplus of premium compared to the  standard modeling with a homogenous Poisson process. {A correct estimate of this surplus of premium is a   crucial challenge for cyber insurance, to avoid an underestimation of the risk induced by a  Poisson process model. Such formula } could  also  be efficient  for sensitivity analysis.  The formula relies on  the
 so-called shifted Hawkes processes, which, from the actuarial point of view, can be seen as "stressed"  scenarios. \\

\noindent From the probabilistic point of view, the quantity $\E[K_T h(L_T)]$ can be expressed  (conditioning with respect to the claims) as $\E\left[\int_{(0,T]} Z_t dH_t F\right]$ where $Z$ is a predictable process and $F:=h(L_T)$ is a functional of the Hawkes process. In  the case where the counting process is a Poisson process (or a Cox process), Malliavin calculus enables one to transform this quantity. More precisely, to simplify the discussion, assume $H$ is an homogeneous Poisson process with intensity $\mu>0$ (in other words the self-exciting kernel $\Phi$ is put to $0$), the Malliavin integration by parts formula allows us to derive that\footnote{Note that strictly speaking this formula is not the Malliavin integration by parts formula on the Poisson space as the stochastic integral $\int_{(0,T]} Z_t dH_t$ is not exactly the divergence of $Z$, which explains why the dual operator on the right-hand side is not exactly the Malliavin difference operator. However, in case of predictable integrators $Z$, the classical integration by parts formula can be reduced to this form which is sufficient for our purpose.} : 
\begin{equation}
\label{eq:IntroIPPPoisson}
\E\left[\int_{(0,T]} Z_t dH_t F\right] = \mu \int_0^T \E\left[Z_t F\circ \eps_t^+\right] dt,
\end{equation}
where the notation $F\circ \eps_t^+$ denotes the functional on the Poisson space where a deterministic jump is added to the paths of $H$ at time $t$. This expression turns out to be particularly interesting from an actuarial point of view since adding a jump at some time $t$ corresponds to realising a stress test by adding artificially a claim at time $t$. This approach has been followed in \cite{HJR2018} for Cox processes (that is doubly stochastic Poisson processes with stochastic but independent intensity). Naturally, in case of a Poisson process, the additional jump at some time $t$ only impacts the payoff of the contract by adding a new claim in the contract but it does not impact the dynamic of the counting process $H$.\\\\
\noindent
The goal of this paper is two-fold: \\\\
1. First we provide in  Theorem \ref{th:mainIPP} a generalization of Equation (\ref{eq:IntroIPPPoisson}) in case $H$ is a Hawkes process. The main ingredient consists in using a representation of a Hawkes process known as the "Poisson embedding" (related to the "Thinning Algorithm") in terms of a Poisson process $N$ on $[0,T]\times\real_+$ to which the Malliavin integration by parts formula can be applied. As the adjunction of a jump at a given time impacts the dynamic of the Hawkes process, we refer to the obtained expression more to an "expansion" rather than an "integration by parts formula" for the Hawkes process, as it involves what we name "shifted Hawkes processes" for which jumps at deterministic times are added to the process accordingly to the self-exciting kernel $\Phi$. We refer to Theorem \ref{th:mainIPP} and to Remark \ref{rem:postmain} for a discussion on this expansion and its link to the one obtained for homogeneous Poisson process.\\\\
\noindent
2. Then, we apply our main result to the specific quantity $\E[K_T h(L_T)]$ which is at the core for determining the premium of a large class of insurance derivatives or risk management instruments. Our main result on that regard is given in Theorem \ref{th:main_I}. As pointed out in the discussion at the beginning of Section \ref{section:LowerUpper}, the shifted processes $H^{v_n,\ldots,v_1}$ (see Definition \ref{definition:multiShift} for a precise statement) appearing  in the form of the premium are of the same complexity than the original Hawkes process $H$. However, they exhibit deterministic jumps at some times $v_1,\ldots,v_n$ which are weighted by correlation factors of the form $\Phi(v_i-v_{i-1})$. In other words, this formula make appears $n$-jumps of the Hawkes process at some deterministic times. This provides an additional input compared to classical estimates for Hawkes processes for obtaining lower or upper bounds of their CDF in terms of the one of a Poisson process for instance. We benefit from this formulation to derive in Proposition \ref{prop:lowerboundBest} and Proposition \ref{prop:upperbound} a lower and an upper bound respectively for the quantity $\E[K_T h(L_T)]$.\\\\ 
\noindent
We proceed as follows. In the next section, we provide general notations and elements of Malliavin calculus on the (classical) Poisson space. In particular, the shift operators (which will play a central role in our analysis) on the Poisson space are introduced. We also  explain the representation of a Hawkes process using the Poisson embedding.  Section \ref{section:Hawkes} provides the derivation of the expansion formula in Theorem \ref{th:mainIPP}. Note that it requires the introduction  and analysis of what we named the shifted Hawkes processes resulting from the shifts on the Poisson space of the original Hawkes process (this material is presented in Section \ref{section:shifted}). Insurance contracts of  interest are presented in Section \ref{section:insurance} together with the main result for the representation of the premium of such contracts in Theorem \ref{th:main_I}. Lower and upper bounds for this premium are presented in Propostion \ref{prop:lowerboundBest} and Proposition \ref{prop:upperbound}, and in Corollaries \ref{col:lowerboundBest} and \ref{cor:upperbound}. Finally, we postponed some technical material in Section \ref{section:tec}.

\section{Elements of stochastic analysis on the Poisson space, Hawkes process and thinning}\label{sec:Poisson2}

This short section  provides  some generalities and elements of stochastic analysis on the Poisson space, and in particular the integration by parts formula for the Poisson process.  Hawkes processes  are also defined and  their representation  through the thinning procedure is presented.\\
Throughout this paper $T>0$ denotes a fixed positive real number. For $X$ a topological space, we set $\mathcal B(X)$ the $\sigma$-algebra of Borelian sets. 

\subsection{Elements of stochastic analysis on the Poisson space}
\label{sec:Poisson}
Let the space of configurations
$$ \Omega^N:=\left\{\omega^N=\sum_{i=1}^{n} \delta_{t_{i},\theta_i}, \; i=1,\ldots,n,\; 0=t_0 < t_1 < \cdots < t_n \leq T, \; \theta_i \in \mathbb{R}_+, \; n\in \mathbb{N}\cap\{+\infty\} \right\}.$$
Each path of a counting process is represented as an element $\omega^N$ in $\Omega^N$ which is a $\mathbb N$-valued measure on $[0,T]\times \mathbb{R}_+$. Let $\mathcal F_T^N$ be the $\sigma$-field associated to the vague topology on $\Omega^N$, and $\P^N$ the Poisson measure on $ \Omega^N$ under which the counting process $N$ defined as  \b{the} canonical process on $ \Omega^N$ as
$$ (N(\omega))([0,t]\times[0,b])(\omega):=\omega([0,t]\times[0,b]), \quad t \in [0,T], \; b \in \mathbb{R}_+,$$
is an homogeneous Poisson process with intensity one (so that $N([0,t]\times[0,b])$ is a Poisson random variable with intensity $ b t$ for any $(t,b) \in [0,T]\times\mathbb{R}_+$). We set $\mathbb F^N:=(\F_t^N)_{t\in [0,T]}$ the natural filtration of $N$, that is $\mathcal{F}_t^N:=\sigma(N( \mathcal T  \times B), \; \mathcal T \subset \mathcal{B}([0,t]), \; B \in \mathcal{B}(\real_+))$. The expectation with respect to $\P^N$ is denoted by $\E[\cdot]$.\\\\
\noindent
One of the main ingredient in our approach will be the integration by parts formula for the Poisson process $N$ and the shift operators defined below.

\begin{definition}[Shift operator]
\label{definition:shift}
We define for $(t,\theta)$ in $[0,T]\times\mathbb{R}_+$ the measurable map  
$$ 
\begin{array}{lll}
\eps_{(t,\theta)}^+ : &\Omega^N &\to \Omega^N\\
&\omega &\mapsto  \eps_{(t,\theta)}^+(\omega),
\end{array}
$$
with $(\eps_{(t,\theta)}^+(\omega))(A) := \omega(A \setminus {(t,\theta)}) + \textbf{1}_{A}(t,\theta), \; A \in \mathcal{B}([0,T]\times \mathbb{R}_+)$ and 
where $$ \textbf{1}_{A}(t,\theta):=\left\lbrace \begin{array}{l} 1, \quad \textrm{if } (t,\theta)\in A,\\0, \quad \textrm{else.}\end{array}\right. $$
\end{definition}

\begin{remark}
\label{rem:Nshift}
Let $(t_0,\theta_0)$ in $(0,T)\times \real_+$, $t_0 < s < t$, $\mathcal T \in \{(s,t), (s,t], [s,t), [s,t]\}$ and $B$ in $\mathcal B(\real_+)$. We have that 
$$ N\circ \eps_{(t_0,\theta_0)}^+ (\mathcal T\times B) = \eps_{(t_0,\theta_0)}^+ (\mathcal T\times B)=N (\mathcal T\times B), \quad \P-a.s..$$
\end{remark}

\begin{lemma}
\label{lemma:mesur}
Let $t$ in $[0,T]$ and $F$ be an $\mathcal F_t^N$-measurable random variable. Let $v > t$ and $\theta_0\geq 0$. It holds that 
$$ F\circ\eps_{(v,\theta_0)}^+ = F, \quad \P-a.s.. $$
\end{lemma}

\begin{proof}
The proof consists in noticing that a $\mathcal F_t^N$-measurable random variable is a functional of $N_{\cdot \wedge t}$. 
\end{proof}
\noindent Similarly, for  any $\omega$ in $\Omega^N$ and $(t,\theta)$ in $[0,T]\times\real_+$, we set the measure  $\eps_{(t,\theta)}^-(\omega)$  defined as
$$(\eps_{(t,\theta)}^-(\omega))(A):=\omega(A\setminus \{(t,\theta)\}), \quad A \in \mathcal B([0,T]\times \real_+).$$
\noindent We conclude this section with the integration by parts formula (ore more specifically the Mecke formula) on the Poisson space (see \cite[Corollaire 5]{Picard_French_96} or \cite{Nualart_Vives_90}).
\begin{prop}[Mecke's Formula]
\label{prop:IPP}
Let $F$ be in $L^1(\Omega,\mathcal F_T^N,\P)$ and $\mathcal Z=(\mathcal Z(t,\theta))_{t\in [0,T], \theta \in \mathbb R_+}$ be a $\mathbb F^N$-adapted process\footnote{Note that only measurability with respect to $\mathcal F^N_T$ is necessary here.} with $\E\left[\int_0^T |\mathcal Z(t,\theta)| dt\right]<+\infty$ and such that 
\begin{equation}
\label{eq:Prev-presque}
\mathcal Z(t,\theta) \circ \eps_{(t,\theta)}^- = \mathcal Z(t,\theta), \quad \P\otimes dt\otimes d\theta, a.e..
\end{equation}
We have that 
$$ \E\left[F \int_{[0,T]\times \real_+} \mathcal Z(t,\theta)N(dt,d\theta)\right] = \E\left[\int_0^T \int_{\real_+} \mathcal Z(t,\theta) (F \circ \eps_{(t,\theta)}^+) dt d\theta\right].$$
\end{prop}

\noindent 
For the expansion formula in Theorem \ref{th:mainIPP}, the Mecke's formula will be applied  for the process  $\mathcal Z(t,\theta) = Z_t  \textbf{1}_{\{\theta \leq \Lambda_t\}}$ where $\Lambda$ will denote the intensity of the Hawkes process (we refer to the proof of Theorem \ref{th:mainIPP} for a more precise relation between this formula and our result).

 \subsection{Representation of Hawkes processes}

We first recall the definition of a Hawkes process. 

\begin{definition}[Standard Hawkes process, \cite{Hawkes}]
\label{def:standardHawkes}
Let $(\Omega,\mathcal F_T,\P,\mathbb F:=(\mathcal F_t)_{t\in [0,T]})$  be a filtered probability space, $\mu>0$ and $\Phi:[0,T] \to \real_+$ be a bounded non-negative map with $\|\Phi\|_1 <1$. A standard Hawkes process $H:=(H_t)_{t\in [0,T]}$ with parameters $\mu$ and $\Phi$ is a counting process such that   
\begin{itemize}
\item[(i)] $H_0=0,\quad \P-a.s.$,
\item[(ii)] its ($\mathbb{F}$-predictable) intensity process is given by
$$\Lambda_t:=\mu + \int_{(0,t)} \Phi(t-s) dH_s, \quad t\in [0,T],$$
that is for any $0\leq s \leq t \leq T$ and $A \in \mathcal{F}_s$,
$$ \E\left[\textbf{1}_A (H_t-H_s) \right] = \E\left[\int_{(s,t]} \textbf{1}_A \Lambda_r dr \right].$$
\end{itemize}
\end{definition}
\noindent This definition can be generalized as follows,  by considering a starting date $v>0$ and allowing starting points  (for the Hawkes process itself and its intensity) that are $\mathcal F_v$-measurable (that is that are known at time $v$).

\begin{definition}[Generalized Hawkes process]
\label{def:Hawkes}
Let $(\Omega,\mathcal F_T,\P,\mathbb F:=(\mathcal F_t)_{t\in [0,T]})$ be a filtered probability space. Let $v$ in $[0,T]$, $h^v$ be a $\mathcal F_v$-measurable random variable with valued in $\mathbb N$, $\mu^v:=(\mu^v(t))_{t\in [v,T]}$ a positive map such that $\mu^v(t)$ is $\mathcal F_v$-measurable for any $t\geq v$, and $\Phi:[0,T] \to \real_+$ be a bounded non-negative map with $\|\Phi\|_1 <1$. A Hawkes process on $[v,T]$ with parameters $\mu^v$, $h^v$ and $\Phi:[0,T] \to \real_+$ is a ($\mathbb{F}$-adapted) counting process $H:=(H_t)_{t\in[v,T]}$ such that 
\begin{itemize}
\item[(i)] $H_v=h^v,\quad \P-a.s.$,
\item[(ii)] its ($\mathbb{F}$-predictable) intensity process is given by
$$\Lambda_t:=\mu^v(t) + \int_{(v,t)} \Phi(t-s) dH_s, \quad t\in [v,T],$$
that is for any $v\leq s \leq t \leq T$ and $A \in \mathcal{F}_s$,
$$ \E\left[\textbf{1}_A (H_t-H_s) \vert \mathcal F_v\right] = \E\left[\int_{(s,t]} \textbf{1}_A \Lambda_r dr \Big\vert \mathcal F_v \right].$$
\end{itemize}
\end{definition}
\noindent  Our main result relies on the following representation of a Hawkes process known as the "Poisson embedding" and related to the "Thinning Algorithm" (see \textit{e.g.} \cite{Bremaud_Massoulie,Costa_etal,Daley_VereJones,Ogata} and references therein). To this end we consider the filtered probability space $(\Omega,\mathcal F_T,\P,\mathbb F)$ as follows 
$$ \Omega:=\Omega^N, \; \mathcal F_T:=\mathcal F_T^N, \quad \mathbb F:=(\mathcal F_t)_{t\in [0,T]}, \; \mathcal F_t:=\mathcal F_t^N, \; t \in [0,T], \quad \P:=\P^N. $$

\begin{theorem}
\label{th:wellposedSDE}
Let $v$ in $[0,T]$ and $(\mu^v(t))_{t\in[v,T]}$ be a non-negative stochastic process such that for any $t\geq v$, $\mu^v(t)$ is a $\mathcal{F}_v^N$-measurable random variable. 
Let in addition $h^v$ be a $\mathcal{F}_v^N$-measurable random with values in $\mathbb{N}$. On the probability space $(\Omega,\mathcal F_T,\P)$, the SDE (\ref{eq:H}) with
\begin{equation}
\label{eq:Hv}
\left\lbrace
\begin{array}{l}
\h^v_t = h^v + \int_{(v,t]} \int_{\real_+} \textbf{1}_{\{\theta \leq \hat\Lambda_s^v\}} N(ds,d\theta),\quad t\in [v,T] \\\\
\hat\Lambda_t^v = \mu^v(t) + \int_{(v,t)} \Phi(t-u) d\hat H^v_u 
\end{array}
\right.
\end{equation}
admits a unique $\mathbb F^N$-adapted solution $\h$. Uniqueness is understood in the strong sense, that is, if $\h^1$, $\h^2$ denote two solutions then 
$$ \P\left[\sup_{t\in [0,T]} |\h_t^1-\h_t^2| \neq 0\right]=0.$$
\end{theorem}
\noindent This result can be deduced from the general construction of point process of Jacod et al \cite{jacod1994}. Nevertheless for seek of completeness, a  direct  proof  for Hawkes process is postponed to Section \ref{section:proofEDS}.

\begin{corollary}\label{constructionHawkes}
Let $\mu \in \mathbb R_+$. We consider $(H,\Lambda)$ the unique solution to SDE 
\begin{equation}
\label{eq:H}
\left\lbrace
\begin{array}{l}
H_t = \int_{(0,t]} \int_{\real_+} \textbf{1}_{\{\theta \leq \Lambda_s\}} N(ds,d\theta),\quad t\in [0,T] \\\\
\Lambda_t = \mu + \int_{(0,t)} \Phi(t-u) d H_u.
\end{array}
\right.
\end{equation}
We set $\mathbb F^H:=(\mathcal F_t^H)_{t\in [0,T]}$ the natural filtration of $H$ (obviously $\mathcal F_t^H \subset \mathcal F_t^N$).  Then $H$ is a standard Hawkes process. 
\end{corollary}

\begin{proof}
Let $0\leq s \leq t \leq T$ and $A$ in $\mathcal F_s^H$.  Using  Lemma \ref{lemma:mesur},  $\mathcal Z(u,\theta)=\textbf{1}_{\{\theta \leq \Lambda_u\}}$  satisfies Relation (\ref{eq:Prev-presque})  (as $\Lambda$ is predictable), one can thus apply Proposition \ref{prop:IPP} (Mecke's Formula) for  $\mathcal Z(u,\theta)=\textbf{1}_{\{\theta \leq \Lambda_u\}}$.
Then, we have (by conditioning with respect to $\mathcal F_s^H$) 
\begin{align*}
&\E\left[1_{A_s} (H_t-H_s)\right]\\
&=\E\left[1_{A_s} \int_{(s,t]} dH_u\right]\\
&=\E\left[1_{A_s} \int_{(s,t]} \int_{\real_+} \textbf{1}_{\{\theta \leq \Lambda_u\}} N(du,d\theta) \right]\\
&=\E\left[1_{A_s} \int_{(s,t]} \int_{\real_+} \textbf{1}_{\{\theta \leq \Lambda_u\}} du d\theta \right]\\
&=\E\left[1_{A_s} \int_{(s,t]} \Lambda_u du\right],
\end{align*}
where we have used again  Lemma \ref{lemma:mesur} to prove that  $1_{A_s}\circ \eps_{(u,\theta)}^+=1_{A_s}$ for $u>s$.
\end{proof}

\section{An expansion formula for Hawkes processes}
\label{section:Hawkes}
Thanks to the previous representation of a Hawkes process using the Poisson embedding,   we derive in this section an expansion formula. The main result is stated in Theorem \ref{th:mainIPP}. It requires an accurate definition and analysis of what we named the shifted Hawkes processes resulting from the shifts on the Poisson space of the original Hawkes process. 
\subsection{The shifted Hawkes processes}
\label{section:shifted}

We now introduce shifted Hawkes processes that is the effect of the shift operators $\eps^+$ on the Hawkes process. As we will see the resulting Hawkes process can also be described in terms of a Poisson SDE.

\begin{definition}[One shift Hawkes process]
\label{definition:OneShift}
 Let $H$ be a standard Hawkes process with initial intensity $\mu >0$ and  bounded excitation function $\Phi : [0,T] \rightarrow \mathbb R_+$ such that $\|\Phi\|_1 <1$. Let $v$ in $(0,T)$. We set 
\begin{equation}
\label{eq:Oneshift}
\left\lbrace
\begin{array}{l}
H_t^{v} = \displaystyle{\textbf{1}_{[0,v)}(t) H_t + \textbf{1}_{[v,T]}(t) \left(H_{v-}^{v} + 1 +\int_{(v,t]} \int_{\real_+} \textbf{1}_{\{\theta \leq \Lambda_s^{v}\}} N(ds,d\theta)\right)},\\\\
\Lambda_t^{v} = \displaystyle{\textbf{1}_{(0,v]}(t) \Lambda_t + \textbf{1}_{(v,T]}(t) \left(\mu^{v,1}(t) + \int_{(v,t)} \Phi(t-u) dH_u^{v}\right)},
\end{array}
\right.
\end{equation}
with $\mu^{v,1}(t):=\mu + \int_{(0,v]} \Phi(t-u) dH_u^{v} =\mu + \int_{(0,v)} \Phi(t-u) dH_u + \Phi(t-v)$.
\end{definition}

\begin{remark}
$(H^v,\Lambda^v)$ is called a shifted (at time $v$) Hawkes process but it is not a Hawkes process as it possesses a jump at the deterministic time $v$. However, it is a Hawkes process on $(0,v)$ (and coincides with the original Hawkes process $(H,\Lambda)$) and is a Hawkes process in the sense of Definition \ref{def:Hawkes} (generalized Hawkes process) on $[v,T]$.  
\end{remark}

\begin{definition}
Let $v$ in $[0,T]$, we set 
$$ 
\begin{array}{lll}
\eps_{(v,\Lambda_v)}^+ : &\Omega^N &\to \Omega^N\\
&\omega &\mapsto (\eps_{(v,\theta)}^+(\omega))_{\theta=\Lambda_v(\omega)}.
\end{array}
$$
\end{definition}

\begin{remark}
\label{rem:oneshift}
Let $v$ in $[0,T]$ and $\theta_0\geq 0$. As $\P[N(\{v\}\times \real_+) >0]=0$, a direct computation gives that
\begin{align*}
H_v \circ \eps_{(v,\theta_0)}^+&=\left(\int_{[0,v]} \int_{\real_+} \textbf{1}_{\{\theta \leq \Lambda_s\}} N(ds,d\theta)\right)\circ \eps_{(v,\theta_0)}^+ \\
&=H_{v-} + \textbf{1}_{\{\theta_0 \leq \Lambda_v\}}, \quad \P-a.s..
\end{align*}
Hence 
\begin{equation}
\label{eq:shifttemp}
\textbf{1}_{\{\theta_0 \leq \Lambda_v\}} \left(H_v \circ \eps_{(v,\theta_0)}^+\right) = \textbf{1}_{\{\theta_0 \leq \Lambda_v\}} \left(H_v \circ \eps_{(v,\Lambda_{v})}^+\right), \; \P-\textrm{a.s.}.
\end{equation}
As Equation (\ref{eq:Hv}) is completely determined by the values $h^v$ and $(\mu^v(t))_{t\in [0,T]}$ (recalling that $\Phi$ does not depend on $v$), we deduce that for any $\theta_0\geq 0$, on the set $\{\theta_0\leq \Lambda_v\}$,
$$ (H_t\circ \eps_{(v,\theta_0)}^+,\Lambda_t\circ \eps_{(v,\theta_0)}^+)_{t\in[v,T]} = (H_t\circ \eps_{(v,\Lambda_v)}^+,\Lambda_t\circ \eps_{(v,\Lambda_v)}^+)_{t\in[v,T]}.$$
This remark leads us to the following lemma.
\end{remark}

\begin{lemma}
\label{lemma:shift}
Let $v$ in $[0,T]$. We have that 
$$ (H\circ \eps_{(v,\Lambda_v)}^+, \Lambda \circ \eps_{(v,\Lambda_v)}^+) = (H^v,\Lambda^v),$$
where $(H^v,\Lambda^v)$ is defined in (\ref{eq:Oneshift}).
\end{lemma}
\begin{proof}
Recall that 
\begin{equation*}
\left\lbrace
\begin{array}{l}
H_t = H_{v} + \int_{(v,t]} \int_{\real_+} \textbf{1}_{\{\theta \leq \Lambda(s)\}} N(ds,d\theta),\quad t\in [0,T] \\\\
\Lambda(t) = \mu + \int_{(0,v]} \Phi(t-u) d H_u + \int_{(v,t)} \Phi(t-u) d H_u.
\end{array}
\right.
\end{equation*}
Thus, by Lemma \ref{lemma:mesur}, $H_t \circ \eps_{(v,\Lambda_v)}^+ = H_t$ for $t<v$ and $\Lambda_t \circ \eps_{(v,\Lambda_v)}^+ = \Lambda_t$ for $t \leq v$. Let $t\geq v$, we have that (recall that the jump times of $N$ shifted by $\Lambda_v$ after time $v$ coincide with those of $N$)
\begin{equation*}
\left\lbrace
\begin{array}{l}
H_t\circ \eps_{(v,\Lambda_v)}^+ = H_{v-} + 1 + \int_{(v,t]} \int_{\real_+} \textbf{1}_{\{\theta \leq \Lambda\circ \eps_{(v,\Lambda_v)}^+(s)\}} N(ds,d\theta),\quad t\in [0,T] \\\\
\Lambda_t\circ \eps_{(v,\Lambda_v)}^+ = \mu + \int_{(0,v)} \Phi(t-u) d H_u + \Phi(t-v) + \int_{(v,t)} \Phi(t-u) d H_u \circ\eps_{(v,\Lambda_v)}^+.
\end{array}
\right.
\end{equation*}
The result follows by uniqueness of the solution to this SDE.
\end{proof}

\noindent We now proceed iteratively to construct a Multi-shifted Hawkes process: the shifts $(v_1, \cdots, v_{n-1})$ being chosen,  the $n^{th}$ shift  is taken in  the interval $]0, v_n[$.

\begin{definition}[Multi-shifted Hawkes process]
\label{definition:multiShift} \mbox{ \quad }\\ 
Let $n \in \mathbb N^*$, and $0 < v_n < v_{n-1} < \cdots  < v_1 < T$. We set 
\begin{equation}
\label{eq:Multishifts}
\left\lbrace
\begin{array}{l}
H_t^{v_n,\ldots,v_1} = \displaystyle{\textbf{1}_{[0,v_n)}(t) H_t + \sum_{i=1}^{n} \textbf{1}_{[v_{i},v_{i-1})}(t) \left(H_{v_i-}^{v_n,\ldots,v_1} + 1 +\int_{(v_i,t]} \int_{\real_+} \textbf{1}_{\{\theta \leq \Lambda_s^{v_n,\ldots,v_1}\}} N(ds,d\theta)\right)},\\\\
\Lambda_t^{v_n,\ldots,v_1} = \displaystyle{\textbf{1}_{(0,v_n]}(t) \Lambda_t + \sum_{i=1}^{n} \textbf{1}_{(v_{i},v_{i-1}]}(t) \left(\mu^{v_i,n}(t) + \int_{(v_i,t)} \Phi(t-u) dH_u^{v_n,\ldots,v_1}\right)},
\end{array}
\right.
\end{equation}
with $\mu^{v_i,n}(t):=\mu + \int_{(0,v_i]} \Phi(t-u) dH_u^{v_n,\ldots,v_1} =\mu + \int_{(0,v_i)} \Phi(t-u) dH_u^{v_n,\ldots,v_1} + \Phi(t-v_i)$.
\end{definition}

\begin{remark}
Note that the process $H^{v_n,\ldots,v_1}$ is not a Hawkes process as it has deterministic jumps at times $v_n,\ldots,v_1$ but it is a generalized Hawkes process on each interval $(v_{i-1},v_i)$.
\end{remark}

\begin{prop}
\label{prop:nshifts}
Let $n \in \mathbb N^*$, and $0 < v_n < v_{n-1} < \cdots  < v_1 < T$. We have that 
$$ (H,\Lambda) \circ \eps_{(v_1,\Lambda_{v_1})}^+ \circ \cdots \circ \eps_{(v_n,\Lambda_{v_n})}^+ = (H^{v_n,\ldots,v_1},\Lambda^{v_n,\ldots,v_1}).$$
\end{prop}
\noindent The proof is postponed to Section \ref{section:tecbis}. 
 We conclude this section on shifted Hawkes processes with the following remark.

\begin{remark}
\label{rem:multishifts}
Let $v<v_n$ and $\theta_0\geq 0$. Following the lines of Remark \ref{rem:oneshift}, we get that 
$$ \textbf{1}_{\{\theta_0 \leq \Lambda_v\}}(H_t^{v_n,\ldots,v_1}\circ \eps_{(v,\theta_0)}^+,\Lambda_t^{v_n,\ldots,v_1}\circ \eps_{(v,\theta_0)}^+)_{t\in[v,T]} =  \textbf{1}_{\{\theta_0 \leq \Lambda_v\}} (H_t^{v_n,\ldots,v_1}\circ \eps_{(v,\Lambda_v)}^+,\Lambda_t^{v_n,\ldots,v_1}\circ \eps_{(v,\Lambda_v)}^+)_{t\in[v,T]}.$$
\end{remark}

\subsection{Expansion formula for the Hawkes process}
{Let $H=(H_t)_{t\in [0,T]}$ be a standard Hawkes process with parameters $\mu >0$ and  bounded  non-negative  kernel $\Phi$ with $\|\Phi\|_1 <1$, solution to the SDE \eqref{eq:H} and build on the Poisson space $(\Omega^N, \mathcal F_T^N, \mathbb P ^N)$, according to Section \ref{sec:Poisson2}.}
In line of the results obtained in the previous section, we can derive the lemma below. 

\begin{lemma}
\label{lemma:shiftint}
Let $F$ be a $\mathcal F_T^N$-measurable random variable and  $v>0$. We have that 
$$ \int_{\real_+} (F\circ \eps_{(v,\theta)}^+) \textbf{1}_{\{\theta \leq \Lambda_v\}} d\theta = \Lambda_v (F\circ \eps_{(v,\Lambda_v)}^+), \quad \P-\textrm{a.s.}.$$ 
\end{lemma}

\begin{proof}
Let $\theta\geq 0$. Remarks \ref{lemma:shift} and \ref{rem:multishifts} entail 
$$\textbf{1}_{\{\theta \leq \Lambda_v\}} (F\circ \eps_{(v,\theta)}^+) = \textbf{1}_{\{\theta \leq \Lambda_v\}} (F\circ \eps_{(v,\Lambda_v)}^+),$$
which concludes the proof.
\end{proof}

\begin{definition}
\label{definition:Fshift}
Let $F$ be a $\mathcal F^N_T$-measurable random variable and $Z$ a $\mathbb F^N$-predictable process. 
We set 
\begin{itemize}
\item[(i)] for $v$ in $(0,T)$,
$$ F^v:=F  \circ \eps_{(v,\Lambda_{v})}^+, \quad Z^{v} := Z\circ \eps_{(v,\Lambda_{v})}^+;$$
\item[(ii)] for $n \geq 2$ and $0 < v_n < v_{n-1} < \cdots < v_1 < T$, 
\begin{equation}
\label{eq:funcshifted}
F^{v_n,\ldots,v_1} := F^{v_{n-1},\ldots,v_1} \circ \eps_{(v_n,\Lambda_{v_n})}^+, \quad Z^{v_n,\ldots,v_2} := Z^{v_{n-1},\ldots,v_2}\circ \eps_{(v_n,\Lambda_{v_n})}^+.
\end{equation}
\end{itemize}
\end{definition}

{\begin{notation}\label{notation:mPhi} 
For  $n=1$,  we set  $ m_\Phi(\Delta^1):=1$ and for $n\geq 2$, 
$$ m_\Phi(\Delta^n):=\int_0^T  \int_0^{v_{1}}\cdots \int_0^{v_{n-1}} \prod_{i=2}^{n} \Phi(v_{i-1}-v_{i}) d v_n\cdots dv_1, $$
\end{notation}
\noindent To ensure the convergence of the forthcoming expansion formula, one will need  the following assumption  
\begin{equation}\label{hyp:mPhi} 
 \lim_{n \rightarrow +\infty } m_\Phi(\Delta^n) =0.
 \end{equation}
Remark that if $\|\Phi\|_\infty <1$,  $m_\Phi(\Delta^n)  \leq \frac{T^n}{n!}$ and Relation \eqref{hyp:mPhi}  is satisfied.
}

\noindent We have now introduced all the ingredients to state  the expansion formula for the Hawkes process. 

\begin{theorem} [Expansion formula for the Hawkes process] 
\label{th:mainIPP}  \mbox{  \quad }\\
Let $F$ be a bounded $\mathcal{F}_T^N$-measurable random variable and $Z=(Z_t)_{t\in [0,T]}$ be a bounded \\
$\mathbb F^H$-predictable process.    Then for any $M \in \mathbb N^*$, 
\begin{align}
\label{eq:nstepformula}
&\E\left[F \int_{[0,T]} Z_t dH_t\right] \nonumber\\
&=\mu \int_0^T \E\left[Z_{v}  F^{v} \right] dv \nonumber\\
& \quad  +\mu \sum_{n=2}^M \int_0^T  \int_0^{v_{1}}  \cdots \int_0^{v_{n-1}} \prod_{i=2}^{n} \Phi(v_{i-1}-v_{i}) \E\left[Z_{v_1}^{v_n,\ldots,v_2} F^{v_n,\ldots,v_1} \right] d v_n\cdots dv_1 \nonumber\\
&\quad +\int_0^T  \int_0^{v_{1}}  \cdots \int_0^{v_{M}} \prod_{i=2}^{M+1} \Phi(v_{i-1}-v_{i}) \E\left[Z_{v_1}^{v_{M+1},\ldots,v_2} F^{v_{M+1},\ldots,v_1} \Lambda_{v_{M+1}} \right] d v_{M+1} \cdots dv_1.
\end{align}
In addition  if Relation \eqref{hyp:mPhi} is satisfied $(\lim_{n \rightarrow +\infty } m_\Phi(\Delta^n) =0)$, we have that 
\begin{align}
\label{eq:nstepformulabis}
&\E\left[F \int_{[0,T]} Z_t dH_t\right] \nonumber\\
&= \mu \int_0^T \E\left[Z_{v}  F^{v} \right] dv \nonumber\\
&+\mu \sum_{n=2}^{+\infty} \int_0^T  \int_0^{v_{1}}  \cdots \int_0^{v_{n-1}} \prod_{i=2}^{n} \Phi(v_{i-1}-v_{i}) \E\left[Z_{v_1}^{v_n,\ldots,v_2} F^{v_n,\ldots,v_1} \right] d v_n\cdots dv_1.
\end{align}
\end{theorem}

\begin{remark}
\label{rem:postmain}
Remark that the first term  $\mu \int_0^T \E\left[Z_{v}  F^{v} \right] dv$ corresponds to the formula for a Poisson process (setting the  self-exciting kernel $\Phi$ at zero). Therefore the sum in the second term can be interpreted as a correcting term due to the self-exciting property of the counting process $H$.  Besides, this formula can be used to provide lower and upper bounds on the premium of insurance contracts. 
\end{remark}

\begin{proof}
Set $\mathcal Z(t,\theta)= Z_t \textbf{1}_{\{\theta \leq \Lambda_t\}}$. As $Z$ is $\mathbb F^H$-predictable, it is $\mathbb F^N$-predictable and thus Relation (\ref{eq:Prev-presque}) is satisfied. 
Thus Mecke's formula for Poisson functionals (see Proposition \ref{prop:IPP}) gives that 
\begin{align*}
&\E\left[F \int_{[0,T]} Z_t dH_t\right]\\
&=\E\left[F \int_{[0,T]} \int_{\real_+} Z_t \textbf{1}_{\{\theta \leq \Lambda_t\}} N(dt,d\theta)\right]\\
&=\E\left[\int_{[0,T]} \int_{\real_+} Z_t (F\circ \eps_{(t,\theta)}^+) \textbf{1}_{\{\theta \leq \Lambda_t\}} dt d\theta\right]\\
&=\int_{[0,T]} \E\left[Z_t \int_{\real_+} (F\circ \eps_{(t,\theta)}^+) \textbf{1}_{\{\theta \leq \Lambda_t\}} d\theta\right]dt\\
&=\int_{[0,T]} \E\left[Z_t (F\circ \eps_{(t,\Lambda_t)}^+) \Lambda_t\right]dt\\
&=\mu \; m_1 +I_1,
\end{align*}
with 
$$m_1:= \int_{[0,T]} \E\left[Z_t (F\circ \eps_{(t,\Lambda_t)}^+) \right]dt,$$
$$I_1:=\int_{[0,T]} \E\left[Z_v (F\circ \eps_{(v,\Lambda_v)}^+) \int_{(0,v)} \Phi(v-u) dH_u \right]dv.$$
Setting $F^v:=F\circ \eps_{(v,\Lambda_v)}^+$,  once again by Proposition \ref{prop:IPP}
we have that \begin{align*}
I_1&=\int_{[0,T]} \E\left[Z_{v_1} F^{v_1} \int_{(0,v_1)} \Phi(v_1-v_2) dH_{v_2} \right]dv_1\\
&=\int_{[0,T]} \E\left[Z_{v_1} F^{v_1} \int_{(0,v_1)} \Phi(v_1-v_2) \int_{\real_+} \textbf{1}_{\{\theta \leq \Lambda_{v_2}\}} N(dv_2,d\theta) \right]dv_1\\
&=\int_{[0,T]} \int_0^{v_1}\E\left[\Phi(v_1-v_2) \int_{\real_+} Z_{v_1}\circ \eps_{(v_2,\theta)}^+ F^{v_1}\circ \eps_{(v_2,\theta)}^+ \textbf{1}_{\{\theta \leq \Lambda_{v_2}\}} d\theta \right]dv_2 dv_1\\
&=\int_{[0,T]} \int_0^{v_1} \Phi(v_1-v_2) \E\left[ Z_{v_1}\circ \eps_{(v_2,\Lambda_{v_2})}^+ F^{v_1}\circ \eps_{(v_2,\Lambda_{v_2})}^+ \Lambda_{v_2}\right]dv_2 dv_1\\
&=\mu \; m_2 + I_2,
\end{align*}
with 
$$ m_2:=\int_{[0,T]} \int_0^{v_1} \Phi(v_1-v_2) \E\left[ Z_{v_1}\circ \eps_{(v_2,\Lambda_{v_2})}^+ F^{v_1}\circ \eps_{(v_2,\Lambda_{v_2})}^+ \right]dv_2 dv_1,$$
and 
$$ I_2:= \int_{[0,T]}\int_0^{v_1}\Phi(v_1-v_2) \E\left[ Z_{v_1}\circ \eps_{(v_2,\Lambda_{v_2})}^+ F^{v_1}\circ \eps_{(v_2,\Lambda_{v_2})}^+ \int_{(0,v_2]} \Phi(v_3-v_2) dH_{v_3} \right]dv_2 dv_1.$$
For $n \geq 2$ we set 
$$ m_n:= \int_0^T \int_0^{v_1}\cdots \int_0^{v_{n-1}} \prod_{i=2}^{n} \Phi(v_{i-1}-v_{i}) \E\left[Z_{v_1}^{v_n,\ldots,v_2} F^{v_n,\ldots,v_1} \right] d v_n\cdots dv_1,$$
and
$$ I_n:= \int_0^T  \int_0^{v_1}\cdots \int_0^{v_{n-1}} \prod_{i=2}^{n} \Phi(v_{i-1}-v_{i}) \E\left[Z_{v_1}^{v_n,\ldots,v_2} F^{v_n,\ldots,v_1} \int_{(0,v_n]} \Phi(v_{n+1}-v_n) dH_{v_{n+1}} \right] d v_n\cdots dv_1.$$
We have that 
\begin{align*}
&\hspace{-3em}I_n \\
&\hspace{-3em}=\int_0^T\int_0^{v_1} \cdots \int_0^{v_{n-1}} \prod_{i=2}^{n} \Phi(v_{i-1}-v_{i}) \E\left[Z_{v_1}^{v_n,\ldots,v_2} F^{v_n,\ldots,v_1} \int_{(0,v_n]} \Phi(v_{n+1}-v_n) \textbf{1}_{\{\theta \leq \Lambda_{v_{n+1}}\}} N(d v_{n+1},d\theta) \right] d v_n\cdots dv_1\\
&\hspace{-3em}=\int_0^T \int_0^{v_1} \cdots \int_0^{v_{n}} \prod_{i=2}^{n+1} \Phi(v_{i-1}-v_{i}) \E\left[\int_{\real_+} (Z_{v_1}^{v_n,\ldots,v_2} F^{v_n,\ldots,v_1})\circ \eps^+ _{(v_{n+1},\theta)} \textbf{1}_{\{\theta \leq \Lambda_{v_{n+1}}\}} d\theta \right] d v_{n+1} \cdots dv_1\\
&\hspace{-3em}=\int_0^T \int_0^{v_1}\cdots \int_0^{v_{n}} \prod_{i=2}^{n+1} \Phi(v_{i-1}-v_{i}) \E\left[(Z_{v_1}^{v_n,\ldots,v_2} F^{v_n,\ldots,v_1})\circ \eps^+ _{(v_{n+1},\Lambda_{v_{n+1}})}\Lambda_{v_{n+1}} \right] d v_{n+1} \cdots dv_1\\
&\hspace{-3em}=\int_0^T\int_0^{v_1} \cdots \int_0^{v_{n}} \prod_{i=2}^{n+1} \Phi(v_{i-1}-v_{i}) \E\left[Z_{v_1}^{v_{n+1},\ldots,v_2} F^{v_{n+1},\ldots,v_1}\Lambda_{v_{n+1}}\right] d v_{n+1} \cdots dv_1\\
&\hspace{-3em}=\mu \; m_{n+1} + I_{n+1}.
\end{align*}
\noindent Hence, by induction, we obtain the first part  of  Theorem \ref{th:mainIPP}  (Equation \eqref{eq:nstepformula}).\\
{ Classical estimates (using the expression of $\Lambda$ in Definition \ref{def:standardHawkes} and \cite[Lemma 3]{Bacry_et_al_2013}) yield  that $ \E[\Lambda_t] \leq \mu \left(1+\frac{\|\Phi\|_1}{1-\|\Phi\|_1}\right)$,  for any $t$}. Therefore, as $Z$ and $F$ are assumed to be bounded, there exists $C>0$ (depending on  $\|Z\|_\infty$,  $\|F\|_\infty$, $\mu$ and $\|\Phi\|_1$) such that 
\begin{align*}
&\int_0^T \int_0^{v_1} \cdots \int_0^{v_{M}} \prod_{i=2}^{M+1} \Phi(v_{i-1}-v_{i}) \E\left[Z_{v_1}^{v_{M+1},\ldots,v_2} F^{v_{M+1},\ldots,v_1}  \Lambda_{v_{M+1}} \right] d v_{M+1} \cdots dv_1\\
&\leq \b{C }    \int_0^T  \int_0^{v_1} \cdots \int_0^{v_{M}} \prod_{i=2}^{M+1} \Phi(v_{i-1}-v_{i})    d v_{M+1} \cdots dv_1= \b{C }  m_\Phi(\Delta^{M+1}) \end{align*}
which converges to zero as $M \rightarrow +\infty $ due to Relation  \eqref{hyp:mPhi}.
\end{proof}
\noindent An alternative form can be given for the representation \eqref{eq:nstepformulabis}. Let $n\geq 1$ and $\Delta^n$ the $n$-dimensional simplex of $[0,T]$ : 
$$ \Delta^n:=\left\{0< v_n < \cdots < v_1< T\right\}.$$
Let $\mathcal U^n$ be a flat Dirichlet distribution on $\Delta^n$ that is a continuous random variable uniformly distributed  on  $\Delta^n$, then for any Borelian integrable map $\varphi:\real^n \to \real$, 
$$ \E\left[\varphi(\mathcal U^n)\right]= \frac{n! }{T^n} \int_0^T \cdots \int_0^{v_{n-1}} \varphi(v_n,\ldots,v_1) dv_n \cdots dv_1.$$

\begin{remark}
Let $(\mathcal U^n)_{n\geq 2}$ be a sequence of independent random variables such that  \\ $\mathcal U^i:=(\mathcal U_1^i,\ldots,\mathcal U_i^i)$ is a Dirichlet distribution on $\Delta^i$. Then 
\begin{align}
\label{eq:nstepformulaalternative}
\E\left[F \int_{[0,T]} Z_t dH_t\right] 
&=\mu \int_0^T \E\left[Z_{v} F^{v} \right] dv \nonumber\\
&\quad + \mu \sum_{n=2}^{+\infty} \frac{T^n}{n!} \E\left[\prod_{i=2}^{n} \Phi(\mathcal U^n_{i-1}-\mathcal U^n_{i}) Z_{\mathcal U^n_1}^{\mathcal U^n_n,\ldots,\mathcal U^n_2} F^{\mathcal U^n_n,\ldots,\mathcal U^n_1}\right] .
\end{align}
\end{remark}

\section{Insurance derivatives for cyber risk}
\label{section:insurance}
The expansion formula can be applied to compute closed formula for   the premium of a class of insurance and financial derivatives (such as reinsurance contracts including generalized Stop-Loss contracts, or CDO tranches) or risk management instruments (like Expected Shortfall), { { to provide generalizations of results that have  been proved in  \cite{HJR2018}  in a Cox model setting. \\
In this section, we  choose to focus  on cyber reinsurance contracts.   Indeed, one important feature of cyber risk  is the presence of accumulation phenomena and contagion, than can not be accurately modeled by Poisson processes (see  the  statistical analysis of  Bessy et al. \cite{bessy2020multivariate}, based on the public Privacy Rights Clearinghouse database).  This means that assuming a Poisson process to model the frequency of cyber risk  may induce an underestimation of the risk, due to a misspecification of  the dependency and self-exciting components of the risk.
  Developing a valuation formula for cyber contracts, taking into account  this self-exciting  feature is thus a crucial challenge for cyber insurance. 
 }

\subsection{The cumulative loss processes and derivatives payoffs}
The so-called cumulative loss process is a key process for risk analysis in insurance and reinsurance. It corresponds to  the cumulative sum of claims amounts, the sum being indexed by a counting process modeling the arrival times of the claims. In standard models, the counting process is assumed to be a Poisson process, that is the claims inter-arrivals are assumed iid with exponential distribution. {Nevertheless,  for cyber risk,   a  Hawkes process modeling is more appropriate, due to the auto-excitation feature of cyber  risk (see  \cite{bessy2020multivariate}, \cite{BALDWIN2017}  or  \cite{PENG2016}).}
In the following, we propose closed-form formula for insurance derivatives, such as stop-loss contracts, for which the cumulative loss process is indexed by a Hawkes process 
$L_t := \sum_{i=1}^{H_t} X_i$
where $(H_t)$ is a Hawkes process and $X_i$ models the $i^{th}$-claim amount, that arrives at the random time $\tau^H_i$.\\

\noindent More precisely, the mathematical framework is the following. {Let $(\Omega^N,\mathcal F^N,\P^N)$  the probability space as defined previously.}
We consider $(\Omega^C,\mathcal F^C,\P^C)$  a probability space on which we define  $(\eta_1,\vartheta_1)$ a  $\real_+^2$-valued random variable (we denote by $\mu$ its distribution) and  $(\eta_i,\vartheta_i)_{i\geq 2}$  {and $(\bar \eta_i, \bar \vartheta_i)_{i\geq 1}$}
that are  independent copies of $(\eta_1,\vartheta_1)$.
We set 
$$ \Omega:=\Omega^N \times \Omega^C, \quad \mathbb F:=(\mathcal F_t)_{t\in [0,T]}, \quad \mathcal F_t:=\mathcal F_t^H \otimes \mathcal F^C, \quad \P:=\P^N\otimes\P^C.$$ 
Note that variables (or processes) defined only on $\Omega^N$ (respectively $\Omega^C$) naturally extend to $\Omega$. In addition, $H$ and the underlying Poisson process $N$ are independent of the variables $\eta_i,\bar{\eta}_i,\vartheta_i,\bar{\vartheta}_i$
($i\geq 1$).\\
\noindent We now define the  cumulative loss processes and the insurance derivatives we consider. One  typical example of  insurance contract is Stop-Loss contract. 
A Stop-Loss contract provides to its buyer protection against losses which are larger than  a given threshold $\underline K$.  
If the $i^{th}$-claim  size is $f(\eta_i)$, then the cumulative loss process is given by $L_t := \sum_{i=1}^{H_t} e^{-\kappa (t-\tau^H_i)}  f(\eta_i)$, where $e^{-\kappa (t-\tau^H_i)}$  is a discount factor. {The process $(L_t)$ is the loss that activates the contract.  
Sometimes the compensation amount 
are not exactly the ones that are computed to activate the reinsurance contract
and may depend on other losses $\vartheta_i$.  The "generalized" loss process is then given by  $K_t := \sum_{i=1}^{H_t}  e^{-\kappa (t-\tau^H_i)}g(\eta_i,\vartheta_i)$ .
For example,  for a stop loss contract, the reinsurance company pays the loss amount above a threshold $\underline K$, up to a given amount $(\bar K - \underline K)$ fixed by the contract.
The payoff of a generalized stop loss contract is  then given by 
\begin{equation}\label{payoff generalized stop loss}
\begin{cases} 0, &\textrm{ if } L_T\leq \underline K\\
K_T-\underline K, &\textrm{ if } \underline K \leq L_T \leq \bar K \\
\bar K  -\underline K, &\textrm{ if } L_T\geq \bar K. \end{cases},\end{equation}
Then, the premium\footnote{Remark that when the risk of the contract is neither hedgeable nor even related to a financial market, the  premium  of the contract  relies on the computation of the expectation of the payoff, under the physical probability measure $\mathbb P$.} of such a contract is 
\begin{equation}\label{decomposed general stop loss} 
\E\left[K_T \ind{L_T> \underline K}\right] -  \underline K \P\left[L_T \in [ \underline K,\bar K]\right] + (\bar K - \underline K) \P\left[L_T \geq \bar K \right]. 
\end{equation}
Therefore one has to compute quantities of the form   $\E\left[ K_T h\left(L_T \right)\right]$. The definitions are gathered below.  }

\begin{definition}[Insurance derivative contract]\footnote{ We put strong a condition of boundedness on the functions $f$, $g$ and $h$ since this condition is in force for actuarial derivatives, but it can be relaxed to a weaker integrability condition}\\
We denote by $(\tau^H_i)_{i \in \mathbb N}$ the jump times of the counting process $H$.
\begin{itemize}
\item[(i)]Given $f:\real_+ \to \real_+$  a bounded deterministic function, the loss process is
\begin{equation}
\label{eq:L}
L_t := \sum_{i=1}^{H_t}  e^{-\kappa (t-\tau^H_i)}f(\eta_i), \quad t\in [0,T].
\end{equation}
\item[(ii)] Given $g:\mathbb{R}_+^2 \to \real_+$ a bounded deterministic function, the generalized loss process is 
\begin{equation}
\label{eq:lossbis}
K_t := \sum_{i=1}^{H_t} e^{-\kappa (t-\tau^H_i)} g(\eta_i,\vartheta_i), \quad t\in [0,T].
\end{equation}
\item[(iii)] 
Let $h:\real_+\to \real_+$ be a bounded deterministic function.  We aim at computing  the expectation of the derivatives payoff, that is  the quantity 
\begin{equation}
\label{eq:generalpayoff}
\E\left[K_T h\left(L_T\right)\right].
\end{equation}
\end{itemize}
\end{definition}

\subsection{General pricing formula}
For the  computation  of  $\E\left[K_T h\left(L_T\right)\right]$, we will rely on the formula in Theorem \ref{th:mainIPP}, by writing the cumulative generalized loss process $K_t$  as an integral with respect to the Hawkes process. Namely,  
$$ K_T =\int_{(0,T]} Z_s dH_s, \quad t \in [0,T]$$
with
$$ Z_s := \sum_{i=1}^{+\infty} g(\eta_i,\vartheta_i)   e^{-\kappa (T-s)}   \textbf{1}_{(\tau_{i-1}^H,\tau_i^H]}(s), \quad s \in [0,T].$$
We first introduce the following definition.
\begin{definition}
Let $n\in \mathbb N^*$ and $0<v_n<\cdots<v_1<T$. \\
We set $\tau^{v_n,\ldots,v_1}_i$ the $i^{th}$-jump time of the shifted Hawkes process $H_{T}^{v_n,\ldots,v_1}$ and 
$$ L_T^{v_n,\ldots,v_1}:= \sum_{i=1}^{H_{T}^{v_n,\ldots,v_1}}   e^{-\kappa (T-\tau^{v_n,\ldots,v_1}_i)}    f(\eta_i) .$$
\end{definition}
\begin{theorem}
\label{th:main_I}
Under the previous assumptions  and Relation  \eqref{hyp:mPhi},  it holds that
\begin{align}
\label{eq:mainSL}
&\E[K_{T} h(L_{T})]\nonumber\\
&\hspace{-1em}=\mu \int_0^T e^{-\kappa (T-v_1)} \E\left[g(\bar{\eta}_1,\bar{\vartheta}_1) \E\left[h\left(e^{-\kappa (T-v_1)} f(\bar{\eta}_1)+\sum_{\underset{ \tau^{v_1}_i \neq v_1}{i=1}}^{(H_T^{v_1})-1} e^{-\kappa (T-\tau^{v_1}_{i})} f(\eta_i)\right) \Big\vert \bar{\eta}_1\right] \right] dv_1\nonumber\\
&\hspace{-1em}+\mu \sum_{n=2}^{+\infty} \int_0^T  \int_0^{v_{1}}   \cdots \int_0^{v_{n-1}} e^{-\kappa (T-v_1)} \prod_{i=2}^{n} \Phi(v_{i-1}-v_{i}) \\
&\E\left[g(\bar{\eta}_1,\bar{\vartheta}_1) \E\left[h\left(\sum_{k=1}^n e^{-\kappa (T-v_k)} f(\bar{\eta}_k)+\sum_{\underset{ \tau^{v_n,\ldots,v_1}_i \neq v_n,\ldots,v_1}{i=1}}^{(H_T^{v_n,\ldots,v_1})-n} e^{-\kappa (T-\tau^{v_n,\ldots,v_1}_{i})} f(\eta_i)\right) \Big\vert \bar{\eta}_1\right] \right] d v_n\cdots dv_1. \nonumber\
\end{align}
where $(\bar{\eta}_i,\bar{\vartheta}_i)_{i\geq 1}$  are independent copies of $(\eta_1,\vartheta_1)$, independent of all  other variables.
\end{theorem}

\begin{proof}
For all $t \in [0,T],$ $ K_t =\int_{(0,t]} Z_s dH_s$
with
$ Z_s := \sum_{i=1}^{+\infty} e^{-\kappa (T-s)}  g(\eta_i,\vartheta_i) \textbf{1}_{(\tau_{i-1}^H,\tau_i^H]}(s).$
Note that $Z_s =  e^{-\kappa (T-s)}     g(\eta_{1+H_{s-}},\vartheta_{1+H_{s-}})$ for any $s$ in $(0,T]$.
Recall that $L_{T}= \sum_{i=1}^{H_{T}} e^{-\kappa (T-\tau_i)}  f(\eta_i)$.\\
 By Theorem \ref{th:mainIPP}, we have that  
\begin{align*}
&\E[K_{T} h(L_{T})]\\
&=\E\left[h(L_{T}) \int_{(0,T]}    Z_t dH_t\right]\\
&=\E\left[h(L_{T}) \int_{(0,T]}  e^{-\kappa (T-t)}  g(\eta_{1+H_{t-}},\vartheta_{1+H_{t-}}) dH_t\right]\\
&=\E\left[\E\left[h(L_{T}) \int_{(0,T]} Z_t dH_t\vert \mathcal F^C\right]\right]\\
&=\mu \int_0^T \E\left[e^{-\kappa (T-v_1)}  g(\eta_{1+H_{v_1-}},\vartheta_{1+H_{v_1-}}) h(L_{T}^{v_1}) \right] dv_1
+\mu  \sum_{n=2}^{+\infty}  I_n 
\end{align*}
with  for $n \geq 2$ $$ I_n:=\int_0^T  \int_0^{v_{1}}   \cdots \int_0^{v_{n-1}} \prod_{i=2}^{n} \Phi(v_{i-1}-v_{i}) \E\left[e^{-\kappa (T-v_1)}  g(\eta_{1+H_{v_1-}^{v_n,\ldots,v_2}},\vartheta_{1+H_{v_1-}^{v_n,\ldots,v_2} }) h(L_{T}^{v_n,\ldots,v_1}) \right] d v_n\cdots dv_1$$
and for $n=1$
$$ I_1:= \int_0^T \E\left[e^{-\kappa (T-v_1)}  g(\eta_{1+H_{v_1-}},\vartheta_{1+H_{v_1-}}) h(L_{T}^{v_1}) \right] dv_1.$$
By definition, $1+H_{v_1-}^{v_n,\ldots,v_2} = H_{v_1}^{v_n,\ldots,v_1}$, $\P$-a.s.. 
Using the fact that the $(\eta_i)$ are iid, we can separate the claim at time $v_1$ (that we will represent using the random variables ($\bar{\eta}_1, \bar{\vartheta}_1$)) from the other claims. Thus 
\begin{align*}
&I_n \\
&= \int_0^T  \int_0^{v_{1}}   \cdots \int_0^{v_{n-1}} \prod_{i=2}^{n} \Phi(v_{i-1}-v_{i}) \\
&\E\left[e^{-\kappa (T-v_1)} g(\bar{\eta}_1,\bar{\vartheta}_1) h\left(e^{-\kappa (T-v_1)} f(\bar{\eta}_1)+\sum_{\underset{ \tau^{v_n,\ldots,v_1}_i \neq v_1}{i=1}}^{H_T^{v_n,\ldots,v_1}-1} e^{-\kappa (T-\tau^{v_n,\ldots,v_1}_{i})} f(\eta_i)\right) \right] d v_n\cdots dv_1\\
&= \int_0^T  \int_0^{v_{1}}   \cdots \int_0^{v_{n-1}} \prod_{i=2}^{n} \Phi(v_{i-1}-v_{i}) \\
&\E\left[e^{-\kappa (T-v_1)} g(\bar{\eta}_1,\bar{\vartheta}_1) \E\left[h\left(e^{-\kappa (T-v_1)} f(\bar{\eta}_1)+\sum_{\underset{ \tau^{v_n,\ldots,v_1}_i \neq v_1}{i=1}}^{H_T^{v_n,\ldots,v_1}-1} e^{-\kappa (T-\tau^{v_n,\ldots,v_1}_{i})}f(\eta_i)\right) \Big\vert (\bar{\eta}_1,\bar{\vartheta}_1)\right] \right] d v_n\cdots dv_1\\
&= \int_0^T  \int_0^{v_{1}}   \cdots \int_0^{v_{n-1}} \prod_{i=2}^{n} \Phi(v_{i-1}-v_{i}) \\
&\E\left[e^{-\kappa (T-v_1)} g(\bar{\eta}_1,\bar{\vartheta}_1) \E\left[h\left(\sum_{k=1}^n e^{-\kappa (T-v_k)} f(\bar{\eta}_k)+\sum_{\underset{ \tau^{v_n,\ldots,v_1}_i \neq v_n,\ldots,v_1}{i=1}}^{(H_T^{v_n,\ldots,v_1})-n} e^{-\kappa (T-\tau^{v_n,\ldots,v_1}_{i})} f(\eta_i)\right) \Big\vert  \bar{\eta}_1\right] \right] d v_n\cdots dv_1.
\end{align*}
The treatment for $I_1$ is similar, which concludes the proof.
\end{proof}

\noindent The  expansion formula  \eqref{eq:mainSL} is written in terms of the shifted Hawkes process  $H^{v_n,\ldots,v_1}$. The  shifted Hawkes process  has three "types" of jumps
\begin{itemize}
\item the spontaneous jumps induced by an homogeneous Poisson with intensity $\mu$
\item the  deterministic enforced jumps  at time $v_n<\cdots < v_1$
\item  the auto-excited jumps that are induced by previous jumps of the process
\end{itemize}
Controlling this different types of jumps allows us to  perform bounds on the premium.

\subsection{Lower and upper bounds}
\label{section:LowerUpper}
We now use Relation (\ref{eq:mainSL}) to  perform  lower and upper bounds on the premium.\\
\noindent {Recall that the premium takes the form $\E[K_T h(L_T)]$ with $K_T = \sum_{i=1}^{H_T} e^{-\kappa (T-\tau_i^H)} X_i$.\\
In  the more simple case without discounting ($\kappa=0$), computing  lower or upper bounds on the premium relies on two types of estimates} : \\
1) estimates on the CDF of a given sum of claims $\sum_{i=1}^n X_i$ (where $n$ is prescribed)\\
2) estimates on the CDF of the counting process $H$. \\
Usually, for insurance derivatives a specific model of claims is considered so that the former quantity $\sum_{i=1}^n X_i$ is accessible. The main issue of course lies in the estimates for the value of $H_T$. For instance, one can make the following type of estimates on the Hawkes process : 
\begin{itemize}
\item[(i)] Obtain an upper bounds on $\P[H_T \geq C]$ for some constant $C$ (by Markov's inequality for instance), but this is only an upper bound. 
\item[(ii)] Get a lower bound on the premium by noting that $H \geq \tilde{N}$ where $\tilde N$ is an homogeneous Poisson process with intensity $\mu$ (as $\Lambda_t \geq \mu$ for any $t\geq 0$).
\end{itemize}   
In both approaches one makes rough estimates on the Hawkes process. Our approach allows for an intermediary situation as the processes $H^{v_n,\ldots,v_1}$ have deterministic jumps at times $v_n,\ldots,v_1$ which are weighted by the kernels $\Phi(v_{i-1}-v_i)$ over the simplex. So we can make less stringent estimates by at least knowing $n$ jumps of the (shifted) Hawkes process. Obviously, on each intervals $[v_{i-1},v_i]$ we have to consider a Hawkes process for which estimates of the form (i)-(ii) above are the best tools available.

\subsubsection{Lower bound}
Using Theorem \ref{th:main_I}, one can already obtain a first lower bound by just considering the $n$ enforced jumps of the shifted Hawkes process and ignoring the other jumps (which is clearly a very rough estimate).  
\begin{prop}

\label{prop:lowerbound}
Assume $h$ is non-decreasing. We have that 
\begin{align}
E[K_{T} h(L_{T})] &\geq  \mu \int_0^T e^{-\kappa (T-v_1)} \E\left[g(\bar{\eta}_1,\bar{\vartheta}_1) h\left(e^{-\kappa (T-v_1)} f(\bar{\eta}_1) \right)\right] dv_1\nonumber\\
&+\mu \sum_{n=2}^{+\infty} \int_0^T \int_0^{v_{1}} \cdots \int_0^{v_{n-1}} e^{-\kappa (T-v_1)} \nonumber\\
&\prod_{i=2}^{n} \Phi(v_{i-1}-v_{i}) \E\left[g(\bar{\eta}_1,\bar{\vartheta}_1) \E\left[h\left(\sum_{k=1}^n e^{-\kappa (T-v_k)} f(\bar{\eta}_k)\right) \Big\vert \bar{\eta}_1\right] \right] d v_n\cdots dv_1. \nonumber\
\end{align}
If $\kappa=0$, the lower bound simplifies as 
\begin{equation}
\label{eq:mainSL_lowerbis}
\E[K_{T} h(L_{T})] \geq \mu \sum_{n=1}^{+\infty} \E\left[g(\bar{\eta}_1,\bar{\vartheta}_1) \E\left[h\left(\sum_{k=1}^n f(\bar{\eta}_k)\right) \Big\vert \bar{\eta}_1\right] \right] m_\Phi(\Delta^n).
\end{equation}
\end{prop}

\begin{proof}
We apply Relation (\ref{eq:mainSL}) and note that for any $n\geq 1$, we have simply used the fact that $H_T^{v_n,\ldots,v_1} \geq n$.
\end{proof}

\noindent The formula above takes only into account the deterministic jumps added to the process. The proposition below is more accurate by  including the jumps of a  homogeneous Poisson process with  constant intensity $\mu$. 
\begin{prop}
\label{prop:lowerboundBest}
Assume $h$ is non-decreasing. We consider a family $(\mathcal U^p)_{p\geq 1}$ of independent random variables (which are constructed from $N$ only), where for each $p$, $\mathcal U^p$ is a flat Dirichlet distributions on $\Delta^p$. Set for $n\geq 1$, $0<v_n<\cdots < v_1 <T$ : 
\begin{align*}
&\alpha_n(v_n,\ldots,v_1)\\
&:=\sum_{p=0}^{+\infty} e^{-(T\mu)} \frac{(T\mu)^p}{p!} \E\left[g(\bar{\eta}_1,\bar{\vartheta}_1)\E\left[h\left(\sum_{k=1}^n e^{-\kappa (T-v_k)} f(\bar{\eta}_k)+\sum_{i=1}^{p} e^{-\kappa (T-{\mathcal U}_{i}^p)} f(\eta_i)\right) \Big\vert \bar{\eta}_1\right]\right].
\end{align*}
It holds that  
\begin{align}
\label{eq:mainSL_lowerbis}
\E[K_{T} h(L_{T})] 
\geq  &\mu \int_0^T e^{-\kappa (T-v_1)} \alpha_1(v_1) dv_1\nonumber\\
&\hspace{-1em}+\mu \sum_{n=2}^{+\infty} \int_0^T \cdots \int_0^{v_{n-1}} e^{-\kappa (T-v_1)} \prod_{i=2}^{n} \Phi(v_{i-1}-v_{i}) \alpha_n(v_n,\ldots,v_1) d v_n\cdots dv_1.
\end{align}
\end{prop}

\begin{proof}
Let $n\geq 1$ and $0<v_n<\cdots <v_1 <T$. By Proposition \ref{prop:nshifts}, $H_T^{v_n,\ldots,v_1} \geq n + H_T, \quad \P-a.s..$
In addition by Lemma \ref{lemma:compHawkes} (see Section \ref{section:teclemma}), we have that $H_T \geq \tilde N_T$, with $\tilde N_t := N([0,t]\times[0,\mu])$ which is an homogeneous Poisson process with intensity $\mu$ and by construction, any jump $\tau_i^{\tilde N}$ of $\tilde N$ different of $v_n,\ldots,v_1$ is a jump of $H_T^{v_n,\ldots,v_1}$ (different of $v_n,\ldots,v_1$). Hence, using Relation (\ref{eq:mainSL}), we have that 
\begin{align*}
&\E[K_{T} h(L_{T})]\\
&\hspace{-1em}\geq \mu \int_0^T e^{-\kappa (T-v_1)} \E\left[g(\bar{\eta}_1,\bar{\vartheta}_1) \E\left[h\left(e^{-\kappa (T-v_1)} f(\bar{\eta}_1)+\sum_{\underset{ \tau^{\tilde N}_i \neq v_1}{i=1}}^{\tilde N_T} e^{-\kappa (T-\tau^{\tilde N}_{i})} f(\eta_i)\right) \Big\vert \bar{\eta}_1\right] \right] dv_1\\
&\hspace{-1em}+\mu \sum_{n=2}^{+\infty} \int_0^T \cdots \int_0^{v_{n-1}} e^{-\kappa (T-v_1)} \prod_{i=2}^{n} \Phi(v_{i-1}-v_{i}) \\
&\E\left[g(\bar{\eta}_1,\bar{\vartheta}_1) \E\left[h\left(\sum_{k=1}^n e^{-\kappa (T-v_k)} f(\bar{\eta}_k)+\sum_{\underset{ \tau^{\tilde N}_i \neq v_n,\ldots,v_1}{i=1}}^{\tilde N_T} e^{-\kappa (T-\tau^{\tilde N}_{i})} f(\eta_i)\right) \Big\vert \bar{\eta}_1\right] \right] d v_n\cdots dv_1.
\end{align*}
For $n \geq 1$ and $0<v_n<\cdots <v_1 <T$, it holds that :
\begin{align*}
&\E\left[h\left(\sum_{k=1}^n e^{-\kappa (T-v_k)} f(\bar{\eta}_k)+\sum_{\underset{ \tau^{\tilde N}_i \neq v_n,\ldots,v_1}{i=1}}^{\tilde N_T} e^{-\kappa (T-\tau^{\tilde N}_{i})} f(\eta_i)\right) \Big\vert \bar{\eta}_1\right]\\
&=\sum_{p=0}^{+\infty} e^{-(T\mu)} \frac{(T\mu)^p}{p!} \E\left[h\left(\sum_{k=1}^n e^{-\kappa (T-v_k)} f(\bar{\eta}_k)+\sum_{\underset{ \tau^{\tilde N}_i \neq v_n,\ldots,v_1}{i=1}}^{p} e^{-\kappa (T-\tau^{\tilde N}_{i})} f(\eta_i)\right) \Big\vert \bar{\eta}_1\right]\\
&=\sum_{p=1}^{+\infty} e^{-(T\mu)} \frac{(T\mu)^p}{p!} \E\left[h\left(\sum_{k=1}^n e^{-\kappa (T-v_k)} f(\bar{\eta}_k)+\sum_{i=1}^{p} e^{-\kappa (T-{\mathcal U}_{i}^p)} f(\eta_i)\right) \Big\vert \bar{\eta}_1\right],
\end{align*}
which concludes the proof. 
\end{proof}

\begin{remark}
A direct approach for exhibiting $p$ jumps of the Hawkes process would require conditioning on $H_T \geq p$ and then to obtain a lower bound for $\P[H_T \geq p]$ by for example $\P[\tilde N_T \geq p]$ (where $\tilde N$ is the homogeneous Poisson process obtained from $N$ with intensity $\mu$) which is what we do. But the advantage of our approach lies in the fact that on top of these jumps of $\tilde N$ we can go further and consider $n$ jumps of the Hawkes process which somehow are really produced by the self-excitation phenomenon as these jumps are weighted by the kernel $\Phi$ along the jump times $v_n,\ldots,v_1$.
\end{remark}
 
\noindent As a corollary to Proposition \ref{prop:lowerboundBest}, we can deduce the following lower bound in case the discounting factor $\kappa$ is equal to $0$.

\begin{corollary}
\label{col:lowerboundBest}
Assume $h$ is non-decreasing and $\kappa=0$. It holds that  
\begin{equation}
\label{eq:mainSL_lowerbis_col}
\E[K_{T} h(L_{T})] \geq \mu \sum_{n=1}^{+\infty} m_\Phi(\Delta^n) \sum_{p=0}^{+\infty} e^{-(T\mu)} \frac{(T\mu)^p}{p!} \E\left[g(\bar{\eta}_1,\bar{\vartheta}_1)\E\left[h\left(\sum_{k=1}^n f(\bar{\eta}_k)+\sum_{i=1}^{p} f(\eta_i)\right) \Big\vert \bar{\eta}_1\right]\right],
\end{equation}
where we recall that $m_\Phi(\Delta^n)$ is defined in Notation \ref{notation:mPhi}.
\end{corollary}

{
\noindent As pointed out in Remark \ref{rem:postmain}, the first term  in the sum indexed by $n$  (that is for $n=1$) corresponds to the formula for a
Poisson process with intensity $\mu$. Therefore the sum of the remaining terms (for $n \geq 2$) corresponds to a lower bound  for the correcting term due to the self-exciting property of the counting process $H$.
This quantity should be at least added to a computation of the premium  based on a standard Poisson process model.\\}

\noindent  {\bf Application for Stop Loss contract with  deductible.}\\
We consider the following setting of Stop Loss contract ($h(x)= \textbf{1}_{x \geq  \underline K}$) with no discounting  ($\kappa=0$), and with  a deductible on the reporting of the claims such that only claims whose amount exceeds a threshold $\underline f$ are reported.
A lower bound for the surplus of premium due to the self-exciting property of the counting process  $H$ is
$$  \mu  \E\left[g(\bar{\eta}_1,\bar{\vartheta}_1) \right]\sum_{n=2}^{+\infty} m_\Phi(\Delta^n) \sum_{ p +n  \geq \left\lfloor {\underline K /\underline f}\right\rfloor +1     }^{+\infty} e^{-(T\mu)} \frac{(T\mu)^p}{p!} $$
where $\mu$ is the intensity of the "spontaneous"  jumps  (induced an homogeneous Poisson with intensity $\mu$) and  $\E\left[g(\bar{\eta}_1,\bar{\vartheta}_1) \right]$ is the mean cost of one claim.
If we assume furthermore a decreasing  excitation kernel $\Phi $,  then  $m_\Phi(\Delta^n) \geq  \Phi(T)^{n-1}   \frac{T^n}{n!} $ and the lower bound becomes
$$  \mu  \E\left[g(\bar{\eta}_1,\bar{\vartheta}_1) \right]\sum_{n=2}^{+\infty} \Phi(T)^{n-1}   \frac{T^n}{n!} \sum_{ p +n  \geq \left\lfloor {\underline K /\underline f}\right\rfloor +1     }^{+\infty} e^{-(T\mu)} \frac{(T\mu)^p}{p!}. $$

\subsubsection{Upper bound}
We now turn to the upper bound.  { We introduce the following  quantities, as in \cite{Guoetal}.}
\begin{prop}
\label{prop:upperbound}
Assume $h$ is non-decreasing and $ \Phi$ is non-increasing. Let $C_2:=\left(\int_0^T \Psi_1(t) dt\right)^2$ with $\Psi_1$ solution to 
$$ \Psi_1(t) = 1 +\int_0^t \Phi(t-s) \Psi_1(s) ds, \quad t \in [0,T] $$
and 
$C_1:= \int_0^T \Psi_2(t) dt$, with $\Psi_2$ solution to 
$$ \Psi_2(t) = \left(\Psi_1(t)\right)^2 + \int_0^t \Phi(s) \Psi_2(t-s) ds, \quad t \in [0,T]. $$
For $n\geq 1$, let also 
$$ c_n:=(\mu+n\Phi(0)) C_1 + (\mu+n\Phi(0))^2 C_2. $$
Finally set for $n\geq 1$ and $0<v_n<\cdots<v_1 <T$
\begin{align*}
&\beta_n(v_n,\ldots,v_1)\\
&:=e^{-T(\mu+n\Phi(0))}  \E\left[g(\bar{\eta}_1,\bar{\vartheta}_1)\E\left[h\left(\sum_{k=1}^n e^{-\kappa (T-v_k)} f(\bar{\eta}_k)\right) \Big\vert \bar{\eta}_1\right]\right]\\
&+\sum_{p=1}^{+\infty} \left(\frac{c_n}{p^2}\wedge 1\right) \E\left[g(\bar{\eta}_1,\bar{\vartheta}_1)\E\left[h\left(\sum_{k=1}^n e^{-\kappa (T-v_k)} f(\bar{\eta}_k)+\sum_{i=1}^{p} f(\eta_i)\right) \Big\vert \bar{\eta}_1\right]\right]
\end{align*}
We have that 
\begin{align}
\label{eq:mainSL_lower}
\E[K_{T} h(L_{T})] \leq & \mu \int_0^T e^{-\kappa (T-v_1)} \beta_1(v_1) dv_1 \nonumber\\
&\hspace{-1em}+ \mu \sum_{n=2}^{+\infty} \int_0^T \cdots \int_0^{v_{n-1}} \prod_{i=2}^{n} \Phi(v_{i-1}-v_{i}) e^{-\kappa (T-v_1)} \beta_n(v_n,\ldots,v_1)d v_n\cdots dv_1.
\end{align}
\end{prop}

\begin{proof}
Fix $n\geq 1$. By construction (see Lemma \ref{lemma:compHawkes_bis}), $H_T^{v_n,\ldots,v_1} \leq H_T^{\mu+n\Phi(0)}+n $, where $H^{\mu+n\Phi(0)}$ denotes the Hawkes process (constructed from $N$) with initial intensity $\mu$ replaced by $\mu+n\Phi(0)$ (in other words, $H^{\mu+n\Phi(0)}$ is solution to Equation (\ref{eq:H}) with initial intensity $\mu+n\Phi(0)$). Denote by $(\tau_i^n)_i$ the jump times of $H_T^{\mu+n\Phi(0)}$ ; {by the thinning procedure,  the jump times $(\tau_i^H)_i$  of $H$ are necessarily  included in  the jump times $(\tau_i^n)_i$  of $H^{\mu+n\Phi(0)}$}. By Lemma \ref{lemma:compHawkes_bis}, we can make the estimate below.
We have that
\begin{align*}
&\E[K_{T} h(L_{T})]\nonumber\\
&\hspace{-1em}\leq \mu \int_0^T \E\left[e^{-\kappa (T-v_1)} g(\bar{\eta}_1,\bar{\vartheta}_1) \E\left[h\left(e^{-\kappa (T-v_1)} f(\bar{\eta}_1) + \sum_{i=1}^{H_T^{\mu+n\Phi(0)}} e^{-\kappa (T-\tau^{n}_i)} f(\eta_i)\right) \Big\vert \bar{\eta}_1\right] \right]  dv_1\\
&\hspace{-1em}+ \mu \sum_{n=2}^{+\infty} \int_0^T \cdots \int_0^{v_{n-1}} \prod_{i=2}^{n} \Phi(v_{i-1}-v_{i}) \\
&\hspace{-1em}\E\left[e^{-\kappa (T-v_1)} g(\bar{\eta}_1,\bar{\vartheta}_1) \E\left[h\left(\sum_{k=1}^n e^{-\kappa (T-v_k)} f(\bar{\eta}_k) + \sum_{i=1}^{H_T^{\mu+n\Phi(0)}} e^{-\kappa (T-\tau^{n}_i)} f(\eta_i)\right) \Big\vert \bar{\eta}_1\right] \right]  d v_n\cdots dv_1\\
&\hspace{-1em}\leq \mu \int_0^T \E\left[e^{-\kappa (T-v_1)} g(\bar{\eta}_1,\bar{\vartheta}_1) \E\left[h\left(e^{-\kappa (T-v_1)} f(\bar{\eta}_1) + \sum_{i=1}^{H_T^{\mu+n\Phi(0)}} \textcolor{red}{\textbf{}} f(\eta_i)\right) \Big\vert \bar{\eta}_1\right] \right]  dv_1\\
&\hspace{-1em}+ \mu \sum_{n=2}^{+\infty} \int_0^T \cdots \int_0^{v_{n-1}} \prod_{i=2}^{n} \Phi(v_{i-1}-v_{i}) \\
&\hspace{-1em}\E\left[e^{-\kappa (T-v_1)} g(\bar{\eta}_1,\bar{\vartheta}_1) \E\left[h\left(\sum_{k=1}^n e^{-\kappa (T-v_k)} f(\bar{\eta}_k) + \sum_{i=1}^{H_T^{\mu+n\Phi(0)}} \textcolor{red}{\textbf{}} f(\eta_i)\right) \Big\vert \bar{\eta}_1\right] \right]  d v_n\cdots dv_1
\end{align*}  
where we used the upper bound $e^{-\kappa (T-\tau^{n}_i)} \leq 1$ to get rid off the unknown jump times $\tau^{n}_i$.\\
Let $n\geq 1$, $0<v_n<\cdots<v_1<T$. We have that 
\begin{align*}
&\E\left[h\left(\sum_{k=1}^n e^{-\kappa (T-v_k)} f(\bar{\eta}_k) + \sum_{i=1}^{H_T^{\mu+n\Phi(0)}}  f(\eta_i)\right) \Big\vert \bar{\eta}_1\right]\\
&=\sum_{p=0}^{+\infty} \E\left[h\left(\sum_{k=1}^n e^{-\kappa (T-v_k)} f(\bar{\eta}_k) + \sum_{i=1}^{p}  f(\eta_i)\right) \Big\vert \bar{\eta}_1\right] \P[H_T^{\mu+n\Phi(0)} = p]\\
&\leq \E\left[h\left(\sum_{k=1}^n e^{-\kappa (T-v_k)} f(\bar{\eta}_k) \right) \Big\vert \bar{\eta}_1\right] \P[H_T^{\mu+n\Phi(0)} = 0]\\
&+\sum_{p=1}^{+\infty} \E\left[h\left(\sum_{k=1}^n e^{-\kappa (T-v_k)} f(\bar{\eta}_k) + \sum_{i=1}^{p}  f(\eta_i)\right) \Big\vert \bar{\eta}_1\right] \P[H_T^{\mu+n\Phi(0)} \geq p]\\
&\leq \E\left[h\left(\sum_{k=1}^n e^{-\kappa (T-v_k)} f(\bar{\eta}_k) \right) \Big\vert \bar{\eta}_1\right] \P[H_T^{\mu+n\Phi(0)} = 0]\\
&+\sum_{p=1}^{+\infty} \frac{1}{p^2} \E\left[h\left(\sum_{k=1}^n e^{-\kappa (T-v_k)} f(\bar{\eta}_k) + \sum_{i=1}^{p} f(\eta_i)\right) \Big\vert \bar{\eta}_1\right] \E\left[\left(H_T^{\mu+n\Phi(0)}\right)^2\right].
\end{align*}
Then by \cite[Proposition 5]{Guoetal}, $\E\left[\left(H_T^{\mu+n\Phi(0)}\right)^2\right] = (\mu+n\Phi(0)) C_1 + (\mu+n\Phi(0))^2 C_2$ and by \cite[Proposition 7]{Guoetal}, $\P[H_T^{\mu+n\Phi(0)} = 0] = e^{-(\mu+n\Phi(0))T}$. The result follows by injecting these estimates in the previous one.
\end{proof}

\noindent Once again, we consider as a corollary the case where $\kappa=0$.

\begin{corollary}
\label{cor:upperbound}
Assume $h$ is non-decreasing, $\kappa=0$ and $\Phi$  non-increasing. We consider $C_1$, $C_2$ and $c_n$ as defined in Proposition \ref{prop:upperbound}.
For $n\geq 1$ set
\begin{align*}
\beta_n:=&e^{-T(\mu+n\Phi(0))}  \E\left[g(\bar{\eta}_1,\bar{\vartheta}_1)\E\left[h\left(\sum_{k=1}^n f(\bar{\eta}_k)\right) \Big\vert \bar{\eta}_1\right]\right]\\
&+\sum_{p=1}^{+\infty} \frac{c_n}{p^2} \E\left[g(\bar{\eta}_1,\bar{\vartheta}_1)\E\left[h\left(\sum_{k=1}^n f(\bar{\eta}_k)+\sum_{i=1}^{p} f(\eta_i)\right) \Big\vert \bar{\eta}_1\right]\right].
\end{align*}
We have that 
\begin{equation}
\label{eq:mainSL_lower_Bis}
\E[K_{T} h(L_{T})] \leq \mu \sum_{n=1}^{+\infty} m_\Phi(\Delta^n) \beta_n,
\end{equation}
where we recall $m_\Phi(\Delta^n)$ is defined in Notation \ref{notation:mPhi}.
\end{corollary}

\begin{remark}
Note that in case of a Poisson process, the sum in the right-hand-side of (\ref{eq:mainSL_lower_Bis}) resumes to the term $n=1$ (since $\Phi \equiv 0$) which is exactly equal to $\E[K_{T} h(L_{T})]$.
\end{remark}

\begin{remark}
One can relax the monotony assumption on $\Phi$ above by allowing a general bounded $\Phi$ map. In that case, the shifts by $n \Phi(0)$ are replaced with $n \Phi^*$ with $\Phi^*:=\sup_{x \in \real_+} \Phi(x)$. 
\end{remark}

\section{Technical material}
\label{section:tec}

\subsection{Proof of Theorem \ref{th:wellposedSDE}}
\label{section:proofEDS}
We prove the existence and uniqueness of  \eqref{eq:Hv}
\begin{equation*}
\left\lbrace
\begin{array}{l}
\h^v_t = h^v + \int_{(v,t]} \int_{\real_+} \textbf{1}_{\{\theta \leq \hat\Lambda_s^v\}} N(ds,d\theta),\quad t\in [v,T] \\\\
\hat\Lambda_t^v = \mu^v(t) + \int_{(v,t)} \Phi(t-u) d\hat H^v_u 
\end{array}
\right.
\end{equation*} 
for $v$ in $[0,T]$,  where $(\mu^v(t))_{t\in[v,T]}$ a non-negative stochastic process such that for any $t\geq v$, $\mu^v(t)$ is a $\mathcal{F}_v^N$-measurable random variable and  $h^v$  a $\mathcal{F}_v^N$-measurable random with values in $\mathbb{N}$.
The proof is composed of two parts : existence and uniqueness.\\\\
\textbf{Existence  }\\
We start with the existence part. We set 
$$ \Lambda_t^{(1)}:=\mu^v(t), \quad t\in [v,T]. $$
We make use the following notation : if $\tau_i^N$ denotes the $i$th jump time of $N$, then there exists a unique element $\theta_i\geq 0$ such that $N(\{(\tau_i^N\}\times\{\theta_i\})=1$ (in other words we denote by $\theta_i$ the mark associated to the $i^{th}$-jump time of $N$). Let
$$ \tau_1^{\h} :=\inf\left\{\tau_i^N \geq v, \; \theta_1 \leq \Lambda_{\tau_i^N}^{(1)}, \quad i\geq 1\right\} \wedge T.$$
In addition, $\tau_1^{\h}$ is a $\mathbb{F}^N$-stopping time, indeed for any $t \in [0,T]$,
\begin{align*}
\{\tau_1^{\h} \leq t\} &=\bigcup_{i\geq 1} \left[\{v\leq \tau_i^N \leq t\} \cap \left\{\theta_i \leq \Lambda_{\tau_i^N}^{(1)}\right\}\right] \in \mathcal{F}_t^N.
\end{align*}
For $i\geq 1$, we set 
$$  \Lambda_t^{(i+1)}:=\Lambda_t^{(i)} \textbf{1}_{v\leq t \leq \tau_i^{\h}} + \Phi(t-\tau_i^{\h}) \textbf{1}_{\tau_i^{\h}< t \leq T}, \quad t\in [0,T],$$
\begin{equation}
\label{eq:jumpalmost}
\tau_{i+1}^{\h} :=\inf\left\{\tau_k^N, \; \tau_k^N > \tau_i^{\h}, \; \theta_k \leq \Lambda_{\tau_k^N}^{(i+1)}, \quad k\geq 1\right\} \wedge T. 
\end{equation} 
By induction, one proves that for any $i\geq 1$, $\tau_i^{\h}$ is a $\mathbb{F}^N$-stopping time and $\Lambda^{(i)}$ is a $\mathbb{F}^N$-predictable stochastic process as a c\`agl\`ad, $\mathbb{F}^N$-adapted process. In addition, by construction, $\Lambda^{(i+1)}$ and $\Lambda^{(i)}$ coincide on $[0,\tau_i^{\h}]$. Furthermore, $\lim_{i\to+\infty} \tau_i^{\h} = T,\; \P-p.s.$ as $\tau_i^{\h} \geq \tau_i^N$, $\P-$a.s. for any $i\geq 1$.
We set : 
\begin{equation}
\label{eq:jumpexpressSDE}
\Lambda_t:= \lim_{i\to+\infty} \Lambda_t^{(i)}, \quad t \in [0,T],
\end{equation}
which is a $\mathbb{F}$-predictable process. We then set : 
$$ \h_t := h^v +\sum_{i=1}^{+\infty} \textbf{1}_{\{t\geq \tau_i^{\h}\}}, \quad t \in [v,T].$$\\\\
\textbf{Uniqueness  }\\
We now turn to the uniqueness of the solution. To do so we give some immediate properties of any solution to (\ref{eq:H}). Consider, ${\h}$ be a solution to (\ref{eq:H}). Then, by definition, we have that 
$$ \{\tau_k^{\h}, \; k\in \mathbb{N}\} \subset \{\tau_k^N, \; k\in \mathbb{N}\}.$$
In addition, by definition, the jump times of ${\h}$ are given by 
$$\tau_{i}^{\h} :=\inf\left\{\tau_k^N, \; \tau_k^N > \tau_{i-1}^{\h}, \; \theta_k \leq \Lambda_{\tau_k^N}, \quad k\geq 1\right\} \wedge T,$$
where 
\begin{equation}
\label{eq:lambdaalmost}
\Lambda_t = \mu^v(t) + \int_{(v,t)} \Phi(t-s) d\h_s, \quad t\in [v,T].
\end{equation}
Consider, ${\h^1}$, ${\h^2}$, two solution processes to (\ref{eq:H}). Denote by $\tau_k^{\h^1}$ (respectively $\tau_k^{{\h^2}}$) the jump times of ${\h^1}$, $\Lambda^{\h^1}$ the intensity function of ${\h^1}$ (respectively $\Lambda^{{\h^2}}$ the one of ${\h^2}$), with  
$$\Lambda_t^{\h^i}=\mu^v(t)+\int_{(v,t)} \Phi(t-s) d{\h^i}_s, \quad t\in [v,T], \quad i\in \{1,2\}.$$
By the previous remark, $\Lambda^{\h^1} = \Lambda^{{\h^2}}$ on $[0,\tau_1^{\h^1} \wedge \tau_1^{{\h^2}})$ and so $\tau_1^{\h^1} = \tau_1^{{\h^2}}$, $\P$-a.s.. Let $\tau_1:=\tau_1^{\h^1}(=\tau_1^{{\h^2}})$. We have thus that $\Lambda^{\h^1}(\tau_1) = \Lambda^{{\h^2}}(\tau_1)$. Let $i\geq 1$. Assume that 
$$ \tau_{i}^{\h^1} = \tau_{i}^{{\h^2}}(=:\tau_{i}), \; {\h^1}={\h^2} \textrm{ on } [0,\tau_i], \; \textrm{ and } \Lambda^{{\h^1}} = \Lambda^{{\h^2}} \textrm{ on } [0,\tau_i], \quad \P-a.s.. $$ 
Then, as the intensity functions agree up to $\tau_i$ and as the jump times of any solution is characterized by (\ref{eq:jumpexpressSDE}), we deduce that $\tau_{i+1}^{\h^1} = \tau_{i+1}^{{\h^2}}(=:\tau_{i+1})$, $\P$-a.s.. Hence, 
$$ {\h^1}={\h^2} \textrm{ on } [0,\tau_i+1], \; \textrm{ and } \Lambda^{{\h^1}} = \Lambda^{{\h^2}} \textrm{ on } [0,\tau_{i+1}], \quad \P-a.s.. $$
Thus, ${\h^1}_t={\h^2}_t$ for any $t$ in $[0,T]$, $\P$-a.s. (as $\lim_{i\to +\infty} \tau_i = T$, $\P-a.s.$).

\subsection{Proof of  Proposition \ref{prop:nshifts}}\label{section:tecbis}

Let $n \in \mathbb N^*$, and $0 < v_n < v_{n-1} < \cdots  < v_1 <T$. We prove that 
$$ (H,\Lambda) \circ \eps_{(v_1,\Lambda_{v_1})}^+ \circ \cdots \circ \eps_{(v_n,\Lambda_{v_n})}^+ = (H^{v_n,\ldots,v_1},\Lambda^{v_n,\ldots,v_1}).$$

\begin{proof}
We set $(H^{v_0},\Lambda^{v_0}):=(H,\Lambda)$. By Lemma \ref{lemma:shift}, $(H^{v_1},\Lambda^{v_1})=(H^0 \circ \eps_{(v_1,\Lambda_{v_1})}^+,\Lambda^0\circ \eps_{(v_1,\Lambda_{v_1})}^+)$.
Let $n\geq 2$ and assume that 
$$(H^{v_{n-2},\ldots,v_1}\circ \eps_{(v_{n-1},\Lambda_{v_{n-1}})}^+,\Lambda^{v_{n-2},\ldots,v_1}\circ \eps_{(v_{n-1},\Lambda_{v_{n-1}})}^+) = (H^{v_{n-1},\ldots,v_1},\Lambda^{v_{n-1},\ldots,v_1}).$$ 
We prove that 
$$(H^{v_{n-1},\ldots,v_1}\circ \eps_{(v_{n},\Lambda_{v_{n}})}^+,\Lambda^{v_{n-1},\ldots,v_1}\circ \eps_{(v_{n},\Lambda_{v_{n}})}^+) = (H^{v_{n},\ldots,v_1},\Lambda^{v_{n},\ldots,v_1}).$$ 
For simplicity of notations, we set : $(H^{n-1},\Lambda^{n-1}):=(H^{v_{n-1},\ldots,v_1},\Lambda^{v_{n-1},\ldots,v_1})$ and $(H^{n-1,+},\Lambda^{n-1,+}):=(H^{v_{n-1},\ldots,v_1}\circ \eps_{(v_n,\Lambda_{v_n})}^+,\Lambda^{v_{n-1},\ldots,v_1}\circ \eps_{(v_n,\Lambda_{v_n})}^+)$. Since 
\begin{equation*}
\left\lbrace
\begin{array}{l}
H_t^{n-1} = \textbf{1}_{[0,v_n)}(t) H_t + \textbf{1}_{[v_n,v_{n-1})}(t) H_t \\
\quad \quad \quad + \displaystyle{\sum_{i=1}^{n-1} \textbf{1}_{[v_{i},v_{i-1})}(t) \left(H_{v_i-}^{n-1} + 1 + \int_{(v_i,t]}\int_{\real_+} \textbf{1}_{\{\theta \leq \Lambda_s^{n-1}\}} N(ds,d\theta)\right)},\\\\
\Lambda_t^{n-1} = \textbf{1}_{(0,v_n]}(t) \Lambda_t + \textbf{1}_{(v_n,v_{n-1}]}(t) \Lambda_t \\
\quad \quad \quad + \displaystyle{\sum_{i=1}^{n-1} \textbf{1}_{(v_{i},v_{i-1}]}(t) \left(\mu + \int_{(0,v_i]} \Phi(t-u) dH_u^{n-1} + \int_{(v_i,t)} \Phi(t-u) dH_u^{n-1}\right)},
\end{array}
\right.
\end{equation*}

\begin{align*}
H_t^{n-1,+}&=\textbf{1}_{[0,v_n)}(t) H_t + \textbf{1}_{[v_n,v_{n-1})}(t) \left(H_{v_n-} + 1 + \int_{(v_n,t]}\int_{\real_+} \textbf{1}_{\{\theta \leq \Lambda_s^{n-1,+}\}} N(ds,d\theta)\right) \\
&\quad + \sum_{i=n-1}^{1} \textbf{1}_{[v_{i},v_{i-1})}(t) \left(H_{v_i-}^{n-1,+} + 1 + \int_{(v_i,t]}\int_{\real_+} \textbf{1}_{\{\theta \leq \Lambda_s^{n-1,+}\}} N(ds,d\theta) \right)\\
&=\textbf{1}_{[0,v_n)}(t) H_t + \sum_{i=n}^{1} \textbf{1}_{[v_{i},v_{i-1})}(t) \left(H_{v_i-}^{n-1,+} + 1 + \int_{(v_i,t]}\int_{\real_+} \textbf{1}_{\{\theta \leq \Lambda_s^{n-1,+}\}} N(ds,d\theta)\right).
\end{align*}
In addition, 
\begin{align*}
\Lambda_t^{n-1,+} &= \textbf{1}_{(0,v_n]}(t) \Lambda_t \\
&\quad + \textbf{1}_{(v_n,v_{n-1}]}(t) \left(\mu + \int_{(0,v_n)} \Phi(t-u) dH_u + \Phi(t-v_n) + \int_{(v_n,t)} \Phi(t-u) dH_u\circ\eps_{((v_n,\Lambda(v_n))}^+\right) \\
&\quad + \sum_{i=1}^{n-1} \textbf{1}_{(v_{i},v_{i-1}]}(t) \left(\mu + \int_{(0,v_i]} \Phi(t-u) dH_u^{n-1,+} + \int_{(v_i,t)} \Phi(t-u) dH_u^{n-1,+}\right)\\
&= \textbf{1}_{(0,v_n]}(t) \Lambda_t + \textbf{1}_{(v_n,v_{n-1}]}(t) \left(\mu + \int_{(0,v_n]} \Phi(t-u) dH_u^{n-1,+} + \int_{(v_n,t)} \Phi(t-u) dH_u^{n-1,+}\right) \\
&+ \sum_{i=1}^{n-1} \textbf{1}_{(v_{i},v_{i-1}]}(t) \left(\mu + \int_{(0,v_i]} \Phi(t-u) dH_u^{n-1,+} + \int_{(v_i,t)} \Phi(t-u) dH_u^{n-1,+}\right)\\
&= \textbf{1}_{(0,v_n]}(t) \Lambda_t + \sum_{i=1}^{n} \textbf{1}_{(v_{i},v_{i-1}]}(t) \left(\mu + \int_{(0,v_i]} \Phi(t-u) dH_u^{n-1,+} + \int_{(v_i,t)} \Phi(t-u) dH_u^{n-1,+}\right).
\end{align*}
Hence, $(H^{n-1,+},\Lambda^{n-1,+})$ solves the same equation than $(H^{v_n,\ldots,v_1},\Lambda^{v_n,\ldots,v_1})$, which concludes the proof. 
\end{proof}

\subsection{Two comparison lemma}
\label{section:teclemma}
We provide two comparison lemma based on the thinning algorithm.
\begin{lemma}
\label{lemma:compHawkes}
Let $\mu>0$ and $(\hat H,\hat \Lambda)$ the unique solution to 
\begin{equation*}
\left\lbrace
\begin{array}{l}
\hat H_t = \int_{(0,t]} \int_{\real_+} \textbf{1}_{\{\theta \leq \hat \Lambda_s\}} N(ds,d\theta),\quad t\in [0,T] \\\\
\hat \Lambda_t = \mu + \int_{(0,t)} \Phi(t-u) d \hat H_u.
\end{array}
\right.
\end{equation*}
Consider the homogeneous Poisson process $\tilde H$  defined as 
$$ \tilde H_t = \int_{(0,t]} \int_{\real_+} \textbf{1}_{\{\theta \leq \mu \}} N(ds,d\theta),\quad t\in [0,T].$$
It holds that 
$$ \hat H_t \geq \tilde H_t, \quad \forall t \in [0,T], \; \P-a.s.. $$
\end{lemma}

\begin{proof}
Let $t$ in $[0,T]$. As $\hat \Lambda_t \geq \mu$ we have that 
\begin{align*}
\hat H_t - \tilde H_t & = \int_{(0,t]} \int_{\real_+} \left(\textbf{1}_{\{\theta \leq \hat \Lambda_s\}} - \textbf{1}_{\{\theta \leq \mu \}} \right) N(ds,d\theta) \\
&=\int_{(0,t]} \int_{\real_+} \textbf{1}_{\{\mu < \theta \leq \hat \Lambda_s\}} N(ds,d\theta) \\
&\geq 0.
\end{align*}
\end{proof}

\begin{lemma}
\label{lemma:compHawkes_bis}
Assume $\Phi$ is non-increasing. Let $n\geq 1$, $0<v_n<\cdots<v_1<T$. Let $(H^{\mu+n\Phi(0)},\Lambda^{\mu+n\Phi(0)})$ the solution to SDE (\ref{eq:H}) with initial intensity $\mu+n\Phi(0)$ instead of $\mu$. Consider as well $(H^{v_n,\ldots,v_1},\Lambda^{v_n,\ldots,v_1})$ the process defined in Definition \ref{definition:multiShift}. It holds that : 
$$ H_t^{v_n,\ldots,v_1} \leq n+ H_t^{\mu+n\Phi(0)}, \quad \forall t \in [0,T],\; \P-a.s..$$
In addition, if $\tau_i^{\mu+n\Phi(0)}$ denotes a jump time of $H^{\mu+n\Phi(0)}$, then 
$$ \P[\exists k \in \{1,\ldots, n\}, \; \tau_i^{\mu+n\Phi(0)} = v_k] = 0$$
and $\{\tau_i^{H^{v_n,\ldots,v_1}}, \; i\geq 1 \} \backslash \{v_1, \cdots, v_n\} \subset \{\tau_i^{H^{\mu+n\Phi(0)}}, \; i\geq 1\}$ almost surely.  
\end{lemma}

\begin{proof}
On $[0,v_n)$, we have that 
$$ H_t^{v_n,\ldots,v_1} = H_t \leq H_t^{\mu+n\Phi(0)} $$
as $$ \Lambda_t^{v_n,\ldots,v_1} = \Lambda_t \leq \Lambda_t^{\mu+n\Phi(0)}, \quad t \in [0,v_n].$$
At $t=v_n$, we have that $H_t^{v_n,\ldots,v_1} = H_t +1 \leq H_t^{\mu+n\Phi(0)} + n$. In addition, 
$$ \Lambda_{v_n+}^{v_n,\ldots,v_1} = \Lambda_{v_n} + \Phi(0) \leq \Lambda_{v_n} + n \Phi(0)  {\leq} \Lambda_{v_n+}^{\mu+n\Phi(0)}.$$
It is important to note that the jump at time $v_n$ for $H^{v_n,\ldots,v_1}$ will impact the self-exciting part of the intensity while for $H^{\mu+n\Phi(0)}$ we have only shifted the baseline intensity. Hence, as $\Phi$ is decreasing, the previous inequality propagates beyond $v_n$. 
The previous inequalities transfer to any interval $[v_i,v_{i-1})$ until time $T$. As $\Lambda_t^{v_n,\ldots,v_1} \leq \Lambda_t^{\mu+n\Phi(0)}$,  using the thinning algorithm,  any non-deterministic jump of $H_t^{v_n,\ldots,v_1}$ is a jump $H^{\mu+n\Phi(0)}$ (the deterministic jumps being $v_n,\ldots,v_1$).   
\end{proof}


\end{document}